\documentclass{article}

\usepackage{pdfpages}
\usepackage{hyperref}
\usepackage{xcolor}
\usepackage[utf8]{inputenc}
\usepackage{amsmath}
\usepackage{amsfonts}
\usepackage{amssymb}
\usepackage{amsthm}
\usepackage{marginnote}
\usepackage{enumerate}
\usepackage{graphicx}
\usepackage[normalem]{ulem}
\usepackage[title, toc, page]{appendix}
\usepackage{longtable}

\newcommand{\defstyle}[1]{\textbf{#1}}
\newcommand{\myprob}[1]{\mathbb P \left[ #1 \right]}

\newcommand{\omid}[1]{\mathbb E \left[ #1 \right]}

\newcommand{\omidCond}[2]{\mathbb E \left[ #1 \left| #2 \right. \right]}

\newcommand{\norm}[1]{\left| #1 \right|}

\newcommand{\nei}[2]{N_{#1} \left(#2\right)}

\newcommand{\identity}[1]{1_{#1}}
\newcommand{\bs}[1]{\boldsymbol{#1}}
\newcommand{\card}[1]{\# #1}

\iftrue
	\newcommand{\todel}[1]{}
	\newcommand{\del}[1]{}
	\newcommand{\invisible}[1]{}
\fi

\newcommand{\egw}{\texttt{EGW}}

\newcommand{\dimMu}[1]{\overline{\mathrm{{udim}}}_M(#1)}
\newcommand{\dimMl}[1]{\underline{\mathrm{{udim}}}_M(#1)}
\newcommand{\dimM}[1]{{\mathrm{{udim}}}_M(#1)}

\newcommand{\dimH}[1]{\mathrm{{udim}}_H(#1)}

\newcommand{\contentH}[2]{\mathcal H^{#1}_{#2}}

\newcommand{\measH}[1]{\mathcal M^{#1}}
\newcommand{\floor}[1]{\left\lfloor #1 \right\rfloor}

\newcommand{\growthl}[1]{ \underline{\mathrm{growth}}\left(#1 \right)}
\newcommand{\growthu}[1]{ \overline{\mathrm{growth}}\left(#1\right)}
\newcommand{\growth}[1]{{ \mathrm{growth}}\left(#1\right)}
\newcommand{\decayl}[1]{ \underline{\mathrm{decay}}\left(#1\right)}
\newcommand{\decayu}[1]{ \overline{\mathrm{decay}}\left(#1\right)}
\newcommand{\decay}[1]{{ \textrm{decay}}\left(#1\right)}

\newcommand{\essinf}{\mathrm{ess\: inf \: }}
\newcommand{\dstar}{{\mathcal D}_*}

\newcommand{\intensity}[2]{\rho_{#1}(#2)}

\newcommand{\rooot}{origin}
\newcommand{\rooted}{pointed}

\newcommand{\densityU}{\overline{d}}
\newcommand{\densityL}{\underline{d}}
\newcommand{\xeq}[1]{{\normalfont (#1)}}

\numberwithin{equation}{section}

\theoremstyle{theorem}
\newtheorem{theorem}{Theorem}[section]
\newtheorem{lemma}[theorem]{Lemma}
\newtheorem{proposition}[theorem]{Proposition}
\newtheorem{corollary}[theorem]{Corollary}
\newtheorem{conjecture}[theorem]{Conjecture}
\newtheorem{problem}[theorem]{Problem}
\newtheorem{guess}[theorem]{Guess}
\theoremstyle{definition}
\newtheorem{definition}[theorem]{Definition}

\newtheorem{example}[theorem]{Example}
\theoremstyle{definition}
\newtheorem{remark}[theorem]{Remark}
\newtheorem{convention}[theorem]{Convention}

\theoremstyle{theorem}

\numberwithin{equation}{section}

\makeatletter
\let\orgdescriptionlabel\descriptionlabel
\renewcommand*{\descriptionlabel}[1]{%
	\let\orglabel\label
	\let\label\@gobble
	\phantomsection
	\edef\@currentlabel{#1}%
	\let\label\orglabel
	\orgdescriptionlabel{#1}%
}

\makeatother

\begin{document}
	\title{Unimodular Billingsley and Frostman Lemmas}
	\author{Fran\c{c}ois Baccelli\footnote{The University of Texas at Austin, baccelli@math.utexas.edu}, Mir-Omid Haji-Mirsadeghi\footnote{Sharif University of Technology, mirsadeghi@sharif.ir}, and Ali Khezeli\footnote{Tarbiat Modares University, khezeli@modares.ac.ir}}

	\maketitle
	
%
%
%

\begin{abstract}
The notions of unimodular Minkowski and Hausdorff dimensions are defined in~\cite{I}
for unimodular random discrete metric spaces. The present paper is focused on the connections
between these notions and the polynomial growth rate of the underlying space. It is shown that 
bounding the dimension is closely related to finding suitable equivariant weight functions (i.e., measures)
on the underlying discrete space.
The main results are unimodular versions of the mass distribution principle, Billingsley's lemma and Frostman's lemma,
which allow one to derive upper bounds on the unimodular Hausdorff dimension from the growth rate of
suitable equivariant weight functions. 
These results allow one to compute or bound both types of unimodular dimensions in a large set of
examples in the theory of point processes, unimodular random graphs, and self-similarity.
Further results of independent interest are also presented, like a version of the max-flow min-cut 
theorem for unimodular one-ended trees and a weak form of pointwise ergodic theorems for all unimodular discrete spaces.
%
\end{abstract}


\section{Introduction}

This paper is a companion to \cite{I}. 
For short, the latter will referred to as Part I. 
The present paper uses the definitions and symbols of Part~I.\del{
A list of notation is provided in Table~\ref{table:symbols} at the end of the paper to ease the reading. }
For cross~referencing the definitions and results of Part~I, the prefix `I.' is used.
For example, Definition~I.3.16 refers to Definition~3.16 in Part~I.

Part I introduced the notion of unimodular random discrete metric space and two notions of
dimension for such spaces, namely the unimodular Min\-kow\-ski dimension (Section I.3.1)
and the unimodular Hausdorff dimension (Section I.3.3).

The present paper is centered on the connections between these dimensions and the \textit{growth rate}
of the space, which is the polynomial growth rate of $\card{N_r(\bs o)}$,
where $N_r(\bs o)$ represents the closed ball of radius $r$ centered at the \rooot{}
and $\card{N_r(\bs o)}$ is the number of points in this ball.
Section \ref{sec:volumeGrowth} is focused on the basic properties of these connections. It is first shown  
that the upper and lower polynomial growth rates of $\card{N_r(\bs o)}$ (i.e., limsup and liminf
of $ {\log(\card{N_r(\bs o)})}/{\log r}$ as $r\rightarrow \infty$) provide upper and lower bound
for the unimodular Hausdorff dimension, respectively. This is a discrete analogue of Billingsley's
lemma (see e.g., \cite{bookBiPe17}). 
A discrete analogue of the \textit{mass distribution principle} is {also} provided, which is
useful to derive upper bounds on the unimodular Hausdorff dimension. In the {Euclidean case (i.e., for point-stationary
point processes equipped with the Euclidean metric)}, it is shown that the unimodular Minkowski dimension
is bounded from above by the polynomial decay rate of $\omid{1/ \card{N_n(\bs o)}}$.
Weighted versions of these inequalities, where a weight is assigned to each point, are also presented. {As a corollary, a weak form of Birkhoff's pointwise ergodic theorem is established for all unimodular discrete spaces.}

{The bounds derived in Section~\ref{sec:volumeGrowth} are fundamental for calculating the unimodular dimensions. These bounds are used in Section~\ref{sec:examples} to complete the examples of Part~I. Some new examples are also presented in Section~\ref{sec:examples} for further illustration of the results.}

Section \ref{sec:frostman} gives a unimodular analogue of Frostman's lemma. Roughly speaking, 
this result states that { there exists a weight function on the points such that the upper bound in the mass distribution principle is sharp.}  
This result {is} a powerful tool to study the unimodular Hausdorff dimension {and is the basis of many of the results of~\cite{III}; e.g., connections to scaling limits, discrete dimension and capacity dimension.} 
In the Euclidean case, another proof of the unimodular Frostman
lemma is provided using a unimodular version of the max-flow min-cut theorem, which is of independent interest.

{It should be noted that the results regarding upper bounds on the dimension (e.g., in the unimodular mass distribution principle, Billingsley lemma and Frostman lemma) are valid for arbitrary gauge functions as well (see Subsection~I.3.8.2). The lower bounds are also valid for gauge functions satisfying the doubling condition. However, these more general cases are skipped for simplicity of reading.}

\section{Connections to Growth Rate}
\label{sec:volumeGrowth}

Let 
$D$ be a discrete space and $o\in D$. The \defstyle{upper and lower (polynomial) growth rates} of $D$ are 
\begin{eqnarray*}
	\growthu{\card{N_r(o)}} &=& \limsup_{r\rightarrow\infty} {\log \card{N_r(o)}}/{\log r},\\
	\growthl{\card{N_r(o)}} &=& \liminf_{r\rightarrow\infty} {\log \card{N_r(o)}}/{\log r}.
\end{eqnarray*}
$D$ has \defstyle{polynomial growth} if $\growthu{\card{N_r(o)}}<\infty$.
If the upper and lower growth rates are equal, the common value is called the \defstyle{growth rate} of $D$. 
Note that for $v\in D$, one has $N_r(o)\subseteq N_{r+c}(v)$ and $N_r(v)\subseteq N_{r+c}(o)$, where $c:=d(o,v)$.
This implies that $\growthu{\card{N_r(o)}}$ and $\growthl{\card{N_r(o)}}$ do not depend on the choice of the point $o$. 

In various situations in this paper, some \textit{weight} in $\mathbb R^{\geq 0}$ can be assigned to each point of $D$. 
In these cases, it is natural to redefine the growth rate by considering the weights; i.e., by replacing $\card{N_r(o)}$
with the sum of the weights of the points in $N_r(o)$. This will be formalized below using the notion of
\textit{equivariant~processes} of Subsection~I.2.5. 
Recall that an equivariant process should be defined for all discrete spaces $D$.
However, if a random \rooted{} discrete space $[\bs D, \bs o]$ is considered,
it is enough to define weights in almost every realization (see Subsection~I.2.5 for more on the matter).
Also, given $D$, the weights are allowed to be random.

\begin{definition}
\label{def:weight}
An \defstyle{equivariant weight function} $\bs w$ is an equivariant process (Definition I.2.5)
with values in $\mathbb R^{\geq 0}$. For all discrete spaces $D$ and $v\in D$, the (random) value $\bs w(v):=\bs w_{D}(v)$
is called the \defstyle{weight} of $v$. Also, for $S\subseteq D$, let  $\bs w(S):=\bs w_D(S):=\sum_{v \in S}\bs w(v).$
\end{definition}
The last equation shows that one could also call $\bs w$ an \textit{equivariant measure}.

Assume $[\bs D, \bs o]$ is a unimodular discrete space (Subsection~I.2.4).
Lemma~I.2.11 shows that $[\bs D, \bs o; \bs w_{\bs D}]$ is a random \rooted{}
marked discrete space and is unimodular. 

In the following, `\textit{$\bs w_{\bs D}(\cdot)$ is non-degenerate
(i.e., not identical to zero) with positive probability}' means that
$\myprob{\exists v\in \bs D: \bs w_{\bs D}(v)\neq 0}>0.$
If $[\bs D, \bs o]$ is unimodular, Lemma~I.2.14 implies
that the above condition is equivalent to $\omid{\bs w(\bs o)}>0.$
Also, `\textit{$\bs w_{\bs D}(\cdot)$ is non-degenerate a.s.}' means 
$\myprob{\exists v\in \bs D: \bs w_{\bs D}(v)\neq 0}=1$.

\subsection{Unimodular Mass Distribution Principle}
\invisible{Later: Intro. Mention the continuum analogue?  Say that in most cases, it is enough to consider non-weighted versions? (i.e., $\bs w:= 1$ and $\bs w(S)=\card{S}$).  }

\begin{theorem}[Mass Distribution Principle]
\label{thm:mdp-simple}
Let $[\bs D, \bs o]$ be a unimodular discrete space. 
\begin{enumerate}[(i)]
\item Let $\alpha,c,M>0$ and assume there exists an equivariant weight function $\bs w$ such that
\del{the weight of the ball {with} center $\bs o$ and radius $r$ satisfies}$\forall r\geq M: \bs w(N_r(\bs o))\leq c r^{\alpha}$, a.s.
Then, $\contentH{\alpha}{M}(\bs D)$ defined in~\xeq{I.3.3} satisfies
$\contentH{\alpha}{M}(\bs D)\geq \frac 1 c \omid{\bs w(\bs o)}.$
\item If in addition\del{ to~\eqref{eq:thm:mdp:assump}}, $\bs w_{\bs D}(\cdot)$ is non-degenerate with positive probability, then
$\dimH{\bs D}\leq\alpha.$
\end{enumerate}
\end{theorem}

\begin{proof}
Let $\bs R$ be an arbitrary {equivariant} covering {such that $\bs R(\cdot)\in \{0\}\cup [M,\infty)$ a.s.}
By the assumption {on $\bs w$}, $\bs R(\bs o)^{\alpha}\geq \frac 1 c {\bs w(N_{\bs R}(\bs o))}$ a.s. Therefore, 	
\begin{equation}
\label{eq:thm:mdp-simple}
\omid{\bs R(\bs o)^{\alpha}}\geq \frac 1{c} \omid{\bs w(N_{\bs R}(\bs o))}.
\end{equation}
Consider the \textit{independent coupling} of $\bs w$ and $\bs R$; i.e., for each deterministic discrete space $G$, choose $\bs w_G$ and $\bs R_G$ independently (see Definition~I.2.8). Then, it can be seen that the pair $(\bs w, \bs R)$ is an equivariant process. So by Lemma~I.2.11, $[\bs G, \bs o; (\bs w, \bs R)]$ is unimodular. Now, the mass transport principle~(I.2.2) can be used for  $[\bs G, \bs o; (\bs w, \bs R)]$.
By letting $g(u,v):= \bs w(v)\identity{\{v\in N_{\bs R}(u)\}}$,
one gets $g^+(\bs o) = \bs w(N_{\bs R}(\bs o))$. Also, 
$g^-(\bs o) = \bs w(\bs o) \sum_{u\in\bs D} \identity{\{\bs o\in N_{\bs R}(u)\}} \geq \bs w(\bs o)$ a.s.,
where the last inequality follows from the fact that $\bs R$ is a covering.
Therefore, the {mass transport principle} implies that
$\omid{\bs w(N_{\bs R}(\bs o))} \geq \omid{\bs w(\bs o)}$
(recall that by convention, $N_{\bs R}(\bs o)$ is the empty set when $\bs R(\bs o)=0$).
So by~\eqref{eq:thm:mdp-simple}, one gets $\omid{\bs R(\bs o)^{\alpha}} \geq \frac 1 c \omid{\bs w(\bs o)}$.
Since this holds for any $\bs R$, one gets that 
$\contentH{\alpha}{M}(\bs D)\geq \frac 1{c} \omid{\bs w(\bs o)}$ and the first claim is proved. 
		
If, with positive probability, $\bs w_{\bs D}(\cdot)$ is non-degenerate, then Lemma~I.2.14 
implies that $\bs w(\bs o)>0$ with positive probability. So $\omid{\bs w(\bs o)}>0$. Therefore,
$\contentH{\alpha}{1}(\bs D)>0$ and the second claim is proved.
\end{proof}

\del{
\begin{example}
Theorem~\ref{thm:mdp-simple}, applied to the counting measure $\bs w\equiv 1$, implies that
$\dimH{\delta \mathbb Z^k}\leq k$ and $\measH{k}(\delta \mathbb Z^k) \geq (2/\delta)^k$.
As already proved in Propositions~I.\ref{I-prop:lattice-Hausdorff} and~I.3.29, equality holds in both.
\end{example}
}

\subsection{Unimodular Billingsley Lemma}
\label{sec:bounds}

The main result of this subsection is Theorem~\ref{thm:billingsley}. 
It is based on Lemmas~\ref{lem:MDPstronger2} and~\ref{lem:lowerboundiid} below. Lemma~\ref{lem:MDPstronger2}
is a stronger version of the mass distribution principle (Theorem~\ref{thm:mdp-simple}).

\begin{lemma}[{An Upper Bound}]
\label{lem:MDPstronger2}
Let $[\bs D, \bs o]$ be a unimodular discrete space and $\alpha\geq 0$.
\begin{enumerate}[(i)]
\item If there exist $c\geq 0$ and $\bs w$ is an equivariant weight function such that 
$
{\limsup_{r\rightarrow\infty}{\bs w(N_r(\bs o))}/{r^{\alpha}} \leq c, \quad \mathrm{a.s.},}
$
then
$
\contentH{\alpha}{\infty}(\bs D) \geq {\frac 1{2^{\alpha}c} \omid{\bs w(\bs o)}}.
$
\item \label{lem:MDPstronger2:2}
In addition, if $\bs w_{\bs D}(\cdot)$ is non-degenerate with positive probability,
then $\dimH{\bs D}\leq\alpha$.
\end{enumerate}
\end{lemma}

\begin{proof}
	Let $c'>c$ be arbitrary. The assumption implies that 
	$ \sup \{r\geq 0: \bs w(N_r(\bs o))>c'r^{\alpha} \}<\infty $ a.s.
	For $m\geq 1$, let $A_m:=\{v\in \bs D: \forall r\geq m: \bs w(N_r(v))\leq c' r^{\alpha} \},$
	which is an increasing sequence of equivariant subsets. So
	\begin{equation}
	\label{eq:limeqp2}
	\lim_{m\rightarrow\infty} \myprob{\bs o\in A_m} {= 1}.
	\end{equation}
	
	Let $\bs R$ be an equivariant covering such that $\bs R(\cdot)\in \{0\}\cup [m,\infty)$ a.s. One has
	\begin{equation}
	\label{eq:lem:MDPstronger2:1}
	\omid{\bs R(\bs o)^{\alpha}}\geq \omid{\bs R(\bs o)^{\alpha} \identity{\{N_{\bs R}(\bs o)\cap A_m \neq \emptyset \}}}.
	\end{equation}
	If $N_{\bs R}(\bs o)\cap A_m \neq \emptyset$, then $\bs R(\bs o)\neq 0$ and hence $\bs R(\bs o)\geq m$. 
        In the next step, assume that this is the case.
	Let $v$ be an arbitrary point in $N_{\bs R}(\bs o)\cap A_m$. By the definition of $A_m$,
        one gets that for all  $r\geq m$, $\bs w(N_r(v))\leq c' r^{\alpha}.$
	Since $N_{\bs R(\bs o)}(\bs o)\subseteq N_{2\bs R(\bs o)}(v)$, it follows that
        $\bs w(N_{\bs R}(\bs o))\leq \bs w(N_{2\bs R(\bs o)}(v))\leq 2^{\alpha}c'\bs R(\bs o)^{\alpha}$.
	Therefore, \eqref{eq:lem:MDPstronger2:1} gives
	\begin{equation}
	\label{eq:lem:MDPstronger2:2}
	\omid{\bs R(\bs o)^{\alpha}}\geq \frac 1{2^{\alpha}c'} \omid{\bs w(N_{\bs R}(\bs o))\identity{\{N_{\bs R}(\bs o)\cap A_m \neq \emptyset \}}}.
	\end{equation}
	
	By letting $g(u,v):= \bs w(v)\identity{\{v\in N_{\bs R}(u)\}} \identity{\{N_{\bs R}(u)\cap A_m \neq \emptyset\}}$,
	one gets that $g^+(\bs o) = \bs w(N_{\bs R}(\bs o))\identity{\{N_{\bs R}(\bs o)\cap A_m \neq \emptyset\}}$.
	Also, since there is a ball $N_{\bs R}(u)$ that covers $\bs o$ a.s., one has
        $g^-(\bs o) \geq \bs w(\bs o)\identity{\{\bs o \in A_m\}}$ a.s.
        Therefore, the mass transport principle~(I.2.2)
	and~\eqref{eq:lem:MDPstronger2:2} imply that
	$
	\omid{\bs R(\bs o)^{\alpha}}\geq \frac 1{2^{\alpha}c'} \omid{\bs w(\bs o)\identity{\{\bs o\in A_m\}}}.
	$
	This implies that 
	$\contentH{\alpha}{m}(\bs D)\geq \frac 1{2^{\alpha}c'} \omid{\bs w(\bs o)\identity{\{\bs o\in A_m\}}}$. 
        Using~\eqref{eq:limeqp2} and letting $m$ tend to infinity gives
        $\contentH{\alpha}{\infty}(\bs D) \geq \frac 1{2^{\alpha}c'} \omid{\bs w(\bs o)}$.
	Since $c'>c$ is arbitrary, the first claim is proved.
	Part~\eqref{lem:MDPstronger2:2} is also proved by arguments similar to those in Theorem~\ref{thm:mdp-simple}.
\end{proof}

\begin{lemma}[{Lower Bounds}]
	\label{lem:lowerboundiid}
	Let $[\bs D, \bs o]$ be a unimodular discrete space,  $\alpha\geq 0$ and $c>0$.
	Let $\bs w$ be an arbitrary equivariant weight function such that $\omid{\bs w(\bs o)}<\infty$.
	\begin{enumerate}[(i)]
		\item \label{part:lem:lowerbound:all}
		If $\exists r_0: \forall r\geq r_0: \bs w(N_r(\bs o))\geq c r^{\alpha}$ a.s.,
		then $\dimMl{\bs D}\geq \alpha$. 
		\item \label{part:lem:lowerbound:liminf}
		If {$\growthl{\bs w(N_r(\bs o))}\geq \alpha$}
		a.s., then $\dimH{\bs D}\geq \alpha$.
		\item \label{part:lem:lowerbound:E} 
		If {$\lim_{\delta \downarrow 0} \liminf_{r\rightarrow\infty} \myprob{\bs w(N_r(\bs o))\leq \delta r^{\alpha}}=0$,}
		then $\dimH{\bs D}\geq \alpha$. 
		\item \label{part:lem:lowerbound:exp}
		If $\decayl{\omid{\mathrm{exp}\left(-\frac{\bs w(N_n(\bs o))}{n^{\alpha}}\right)}}\geq \alpha$,
                then $\dimMl{\bs D}\geq \alpha$.
	\end{enumerate}
\end{lemma}

\begin{proof}
	The proofs of the first two parts are very similar. The second part is proved first.
	
	\eqref{part:lem:lowerbound:liminf}.
	Let $\beta$, $\gamma$ and $\kappa$ be such that $\gamma < \beta < \kappa < \alpha$. 
	Fix $n\in\mathbb N$. Let $\bs S= \bs S_{\bs D}$ be the equivariant subset obtained by selecting
	each point $v\in\bs D$ with probability {$1\wedge (n^{-\beta}\bs w(v))$} (the selection variables are assumed to be conditionally independent given $[\bs D, \bs o; \bs w]$). Let $\bs R_n(v)=n$ if $v\in \bs S_{\bs D}$,
	$\bs R_n(v)=1$ if $N_n(v)\cap \bs S_{\bs D}=\emptyset$, and $\bs R_n(v)=0$ otherwise.
	Then $\bs R_n$ is an {equivariant} covering. It is shown below that 
	$\omid{\bs R_n(\bs o)^{\gamma}}\rightarrow 0$.
	Let $M:={\sup}\{r\geq 0: \bs w(N_r(\bs o))< r^{\kappa}\}$.
	By the assumption, $M<\infty$ a.s.  One has 
	\begin{eqnarray*}
		\omid{\bs R_n(\bs o)^{\gamma}} &=& n^{\gamma}\myprob{\bs o\in \bs S_{\bs D}} + \myprob{N_n(\bs o)\cap \bs S_{\bs D}=\emptyset}\\
		&=&	n^{\gamma}{\omid{1\wedge n^{-\beta}\bs w(\bs o)}} + \omid{\prod_{v\in N_n(\bs o)}\left(1-(1\wedge {n^{-\beta}}{\bs w(v)})\right)}\\
		&\leq & {n^{\gamma-\beta}\omid{\bs w(\bs o)} + \omid{\exp\left(-n^{-\beta}{\bs w(N_n(\bs o))}  \right)}} \\
		&=& n^{\gamma-\beta}\omid{\bs w(\bs o)} + \omidCond{{\exp\left(-n^{-\beta}{\bs w(N_n(\bs o))}  \right)}}{M<n}\myprob{M<n} \\ & & \quad\quad\quad\quad\quad\quad\hspace{0.2mm}  + \, \omidCond{{\exp\left(-n^{-\beta}{\bs w(N_n(\bs o))}  \right)}}{M\geq n}\myprob{M\geq n}\\
		&\leq & n^{\gamma-\beta}\omid{\bs w(\bs o)} + \exp\left(-n^{\kappa-\beta} \right) + \myprob{M\geq n},
	\end{eqnarray*}
	{where the first inequality holds because $1-(1\wedge x)\leq e^{-x}$ for all $x\geq 0$.}
	Therefore, $\omid{\bs R_n(\bs o)^{\gamma}}\rightarrow 0$ when $n\to\infty$. It follows that $\dimH{\bs D}\geq \gamma$.
	Since $\gamma$ is arbitrary, this implies $\dimH{\bs D}\geq \alpha$.
	
	\eqref{part:lem:lowerbound:all}. Only a small change is needed in the above proof.
	For $n\geq r_0$, let $\bs R_n(v)=n$ if either $v\in \bs S_{\bs D}$ or $N_n(v)\cap \bs S_{\bs D}=\emptyset$,
	and let $\bs R_n(v)=0$ otherwise. Note that $\bs R_n$ is a covering by balls {of equal radii}. 
	By the same computations and the assumption  {$M\leq r_0$}, one gets
	\[
	\myprob{\bs R_n(\bs o)\neq 0} \leq {n^{-\beta}\omid{\bs w(\bs o)} + \exp\left(-n^{\kappa-\beta}\right),}
	\]
	which is of order $n^{-\beta}$ for large $n$.
	This implies that $\dimMl{\bs D}\geq \beta$.
	Since $\beta$ is arbitrary, one gets $\dimMl{\bs D}\geq \alpha$ and the claim is proved.
	
	\eqref{part:lem:lowerbound:E}. 
	Let $\beta<\alpha$. It will be proved below that under the assumption of \eqref{part:lem:lowerbound:E}, there is a sequence
	$r_1,r_2,\ldots$ such that $\omid{\exp\left(-r_n^{-\beta}{\bs w(N_{r_n}(\bs o))}  \right)}\rightarrow 0$.
	If so, by a slight modification of the proof of part~\eqref{part:lem:lowerbound:liminf},
	one can find a sequence of equivariant coverings $\bs R_n$ such that
	$\omid{\bs R_n(\bs o)^{\beta}}<\infty$ and \eqref{part:lem:lowerbound:E} is proved.
	
	Let $\epsilon>0$ be arbitrary. By the assumption, there is $\delta>0$
	and $r\geq 1$ such that $\myprob{\bs w(N_r(\bs o))\leq \delta r^{\alpha}}<\epsilon$. So
	\begin{eqnarray*}
		\omid{\exp\left(-r^{-\beta}{\bs w(N_{r}(\bs o))}  \right)} &\leq& \omidCond{\exp\left(-r^{-\beta}{\bs w(N_{r}(\bs o))}\right)}{\bs w(N_r(\bs o))>\delta r^{\alpha}}\\
		&& + \myprob{\bs w(N_r(\bs o))\leq \delta r^{\alpha}}\\
		&\leq & \mathrm{exp}(-\delta r^{\alpha-\beta}) + \epsilon.
	\end{eqnarray*}
	Note that for fixed $\epsilon$ and $\delta$ as above, $r$ can be arbitrarily large.
	Now, choose $r$ large enough for the right hand side to be at most $2\epsilon$.
	This shows that $\omid{\exp\left(-r^{-\beta}{\bs w(N_{r}(\bs o))}  \right)}$ can be arbitrarily small and the claim is proved.

	\eqref{part:lem:lowerbound:exp}.
	As before, let $\bs R_n(v)=n$ if either $v\in \bs S_{\bs D}$ or $N_n(v)\cap \bs S_{\bs D}=\emptyset$,
	and let $\bs R_n(v)=0$ otherwise. The calculations in the proof of part~\eqref{part:lem:lowerbound:liminf} show that 
	\[
	\myprob{\bs R_n(\bs o)\neq 0} \leq n^{-\beta}\omid{\bs w(\bs o)} + \omid{\exp\left(-n^{-\beta}\bs w(N_n(\bs o))\right)}.
	\]
	Now, the assumption implies the claim.
\end{proof}

\begin{remark}
The assumption in part~\eqref{part:lem:lowerbound:E} of Lemma~\ref{lem:lowerboundiid} is
equivalent to the condition that 
{there exists a sequence $r_n\to \infty$ such that the family of random variables $r_n^{\alpha}/\bs w(N_{r_n}(\bs o))$ is tight.}
Also, 
from the proof of the lemma, one can see that this assumption is equivalent to
\[\liminf_{n\rightarrow\infty} \omid{\mathrm{exp}\left(-\frac{\bs w(N_n(\bs o))}{n^{\alpha}}\right)} =0.\] 
\end{remark}

\begin{theorem}[Unimodular Billingsley Lemma] 
\label{thm:billingsley}
Let $[\bs D, \bs o]$ be a unimodular discrete metric space. {Then,} for all equivariant weight functions $\bs w$ such that $0<\omid{\bs w(\bs o)}<\infty,$ one has
\begin{eqnarray*}
	\essinf \left(\growthl{\bs w(N_r(\bs o))}\right) 
	&\leq& \dimH{\bs D} \\ 
	&\leq& \essinf \left(\growthu{\bs w(N_r(\bs o))}\right)\\
	&\leq& \growthu{\omid{\bs w(N_r(\bs o))}}.
\end{eqnarray*}
\end{theorem}
{The proof is given below. In fact, in many examples, it is enough to consider $\bs w\equiv 1$ in Billingsley's lemma; i.e., $\bs w(N_r(\bs o))=\card{N_r(\bs o)}$.  
}

\begin{corollary}
	\label{cor:billingsley-ergodic}
Under the assumptions of Theorem~\ref{thm:billingsley}, 
if the upper and lower growth rates of $\bs D$ are almost surely constant
(e.g., when $[\bs D,\bs o]$ is \textit{ergodic}), then, 
\begin{eqnarray}
\label{eq:billingsly-simple}
{\growthl{\bs w(N_r(\bs o))}} \leq &{\dimH{\bs D}} &\leq \growthu{\bs w(N_r(\bs o))} \quad a.s.
\end{eqnarray}
{In particular,} if $\growth{\bs w(N_r(\bs o))}$ exists and is constant a.s., then 
\[\dimH{\bs D} = \growth{\bs w(N_r(\bs o))}.\]
\end{corollary}

{In fact, without the assumption of this corollary, an inequality similar to~\eqref{eq:billingsly-simple} is valid} for the \textit{sample Hausdorff dimension} of $\bs D$,
which will be studied in~\cite{III}.

\begin{proof}[Proof of Theorem~\ref{thm:billingsley}]
	\del{The first part is implied by the second one. So it is enough to prove the second part.
		
	}The first inequality is implied by part~\eqref{part:lem:lowerbound:liminf} of Lemma~\ref{lem:lowerboundiid}. 
	The last inequality is implied by Lemma~\ref{lem:logbound}. 
	For the second inequality, assume that ${\growthu{\bs w(N_r(\bs o))}}<\alpha$
	with positive probability.
	On this event, one has $\bs w(N_r(\bs o))\leq r^{\alpha}$ {for large $r$}; i.e., $\limsup_r \bs w(N_r(\bs o))/r^{\alpha}\leq 1$. Now, Lemma~\ref{lem:MDPstronger2} implies that
	$\dimH{\bs D}\leq \alpha$. This proves the result.
\end{proof}

\begin{remark}
\label{rem:bil-nec}
In fact, the assumption $\omid{\bs w(\bs o)}<\infty$ in Theorem~\ref{thm:billingsley}
is only needed for the lower bound while the assumption $\omid{\bs w(\bs o)}>0$ is only needed for the upper bound.
These assumptions are also necessary as shown below.
	\\	
For example, assume {$\Phi$} is a point-stationary point process in $\mathbb R$ (see Example~I.2.6).
For $v\in \Phi$, let $\bs w(v)$ be the sum of the distances of $v$ to its next and previous points in $\Phi$.
This equivariant weight function satisfies $\bs w(N_r(v))\geq 2r$ for all $r$, and hence
$\growthl{\bs w(N_r(\bs o))}\geq 1$. But $\dimH{\Phi}$ can be strictly less than 1 as shown in Subsection~\ref{subsec:image}.
	\\
Also, the condition that $\bs w_{\bs D}$ is non-degenerate a.s. is trivially necessary for the upper bound.
\end{remark}

\del{\begin{remark}
In many examples, it is enough to consider $\bs w\equiv 1$ in Billingsley's lemma
(i.e., $\bs w(N_r(\bs o))=\card{N_r(\bs o)}$). Examples where other weight functions
are used are 2-ended trees (Subsection~\ref{subsec:two-ended}), point-stationary point processes
(Proposition~\ref{prop:upperbound-Rd}), and embedded spaces (Subsection~\ref{subsec:embedded}).
\end{remark}
}

\del{\begin{remark}
		\mar{In fact, Frostman is a converse to MDP. I moved this to the previous subsection.}
A converse to the unimodular Billingsley lemma is \textit{the unimodular Frostman lemma}, 
which will be discussed in Section~\ref{sec:frostman}.
\end{remark}
}

\del{\begin{remark}
		\mar{Mentioned this after the corollary without a remark.}
Without the assumption of part~\eqref{thm:billingsley:1} of Theorem~\ref{thm:billingsley},
the claim is still valid for the \textit{sample Hausdorff dimension} of $\bs D$,
which will be discussed in~\cite{I}.
\end{remark}}


\begin{corollary}
\label{cor:graphs-lowerbound}
Let $[\bs G, \bs o]$ be a unimodular random graph equipped with the graph-distance metric.
If $\bs G$ is infinite almost surely, then $\dimMl{\bs G}\geq 1$ and else, $\dimM{\bs G}=\dimH{\bs G}=0$.
\end{corollary}

\begin{proof}
If $\bs G$ is infinite a.s., then for $\bs w_{\bs G}\equiv 1$, one has $\bs w(N_r(\bs o))\geq r$ for all $r$.
So part~\eqref{part:lem:lowerbound:all} of Lemma~\ref{lem:lowerboundiid} implies the first claim.
{The second claim is implied by Example~I.3.17 (this can be deduced from the unimodular Billingsley lemma as well).}
\del{For all discrete spaces $G$, let $\bs w_G(\cdot)\equiv 1$ if $G$ is finite and $\bs w_G(\cdot)\equiv 0$ otherwise.
One has $\growthu{\bs w(N_r(\bs o))}=0$. If $\bs G$ is finite with positive probability, 
then $\omid{\bs w(\bs o)}>0$. Therefore, the unimodular Billingsley lemma (Theorem~\ref{thm:billingsley})
implies that $\dimH{\bs G}=0$, which in turn implies the second claim.}
\end{proof}

\begin{corollary}
	The unimodular Minkowski and Hausdorff 
	dimensions of any unimodular
	two-ended tree are equal to one.
\end{corollary}
This result has already been shown in Theorem~I.4.1, but can also be deduced from the unimodular Billingsley lemma directly. For this, let $\bs w(v)$ be 1 if $v$ belongs to the trunk of the tree and 0 otherwise.

\begin{problem}
	In the setting of Corollary~\ref{cor:billingsley-ergodic}, is it always the case that $\dimH{\bs D}=\growthu{\bs w(N_r(\bs o))}$?
\end{problem}
The claim of this problem holds in all of the examples in which both quantities are computed in this work. This problem is a corollary of Problem~\ref{prob:growth} and the unimodular Frostman lemma (Theorem~\ref{thm:frostmanGeneral}) below. Note that there are examples where $\growthl{\cdot}\neq \growthu{\cdot}$ as shown in 
Subsections~\ref{subsec:canopy-generalized} and~\ref{subsec:digits}.


%
%

\subsection{Bounds for Point Processes}
\label{subsec:euclidean}


The next results use the following equivariant covering.
Let $\varphi$ be a discrete subset of $\mathbb R^k$
equipped with the $l_{\infty}$ metric and $r\geq 1$.
Let $C:=C_r:=[0,r)^k$, $\bs U:=\bs U_r$ be a point chosen uniformly at random in $-C$,
and consider the partition $\{C+\bs U+z: z\in r\mathbb Z^k\}$ of $\mathbb R^k$ {by cubes}.
Then, for each $z \in r\mathbb Z^k$, choose a random element in $(C+\bs U+z)\cap \varphi$ independently
(if the intersection is nonempty). The distribution of this random element should depend on the set 
$(C+\bs U+z)\cap \varphi$ in a translation-invariant way (e.g., choose with the uniform distribution
or choose the least point in the lexicographic order).
Let ${\bs R=}\bs R_{\varphi}$ assign the value  $r$ to the selected points and zero to the other points of $\varphi$.
\del{As in Example~I.\ref{I-ex:coveringR}, one can show that }{Then,} $\bs R$ is an equivariant covering. 
Also, each point is covered at most $3^k$ times. So $\bs R$ is $3^k$-bounded (Definition~I.3.9).

\begin{theorem}[Minkowski Dimension in the Euclidean Case]
\label{thm:lowerBoundR^d}
Let $\Phi$ be a point-stationary point process in $\mathbb R^k$
and assume the metric in $\Phi$ is equivalent to the Euclidean metric.
Then, for all equivariant weight functions $\bs w$ such that ${\bs w_{\Phi}(0)}>0$ a.s., one has 
\begin{eqnarray*}
\dimMu{\Phi}  = \decayu {\omid{\bs w(0)/ \bs w(C_r+\bs U_r)}} &\leq& \decayu{\omid{\bs w(0)/ \bs w(N_r(0))}}\\
& \leq & \growthu{\omid{\bs w(N_r(0))}},\\
\dimMl{\Phi}  = \decayl {\omid{\bs w(0)/ \bs w(C_r+\bs U_r)}} &\leq& \decayl{\omid{\bs w(0)/ \bs w(N_r(0))}}\\
& \leq & \growthl{\omid{\bs w(N_r(0))}},
\end{eqnarray*}
where $\bs U_r$ is a uniformly at random point in $-C_r$ independent of $\Phi$ and $\bs w$. 
\end{theorem}

\begin{proof}
By Theorem~I.3.31, one may assume the metric on $\Phi$ is the $l_{\infty}$ metric without loss of generality. 
Given any $r>0$, consider the equivariant covering $\bs R$ described above, but when choosing a random 
element of $(C_r+\bs U_r+z)\cap \varphi$, choose point $v$ with probability
$\bs w_{\varphi}(v)/\bs w_{\varphi}(C_r+\bs U_r+z)$ (conditioned on $\bs w_{\varphi}$). One gets
	$
	\myprob{0 \in \bs R} = \omid{ {\bs w(0)}/{\bs w(C_r+\bs U_r)}}.
	$
As mentioned above,
$\bs R$ is equivariant and uniformly bounded (for all $r>0$).
So Lemma~I.3.10 implies both equalities in the claim.
The inequalities are implied by the facts that $\bs w(C_r+\bs U_r)\leq \bs w(N_r(0))$ and
	\[
	\omid{\frac {\bs w(0)}{\bs w(\nei{r}{0})}} \omid{\bs w(\nei{r}{0})} \geq \omid{\sqrt{\bs w(0)}}^2>0,
	\]
	which is implied by the Cauchy-Schwartz inequality.
\end{proof}

\del{\mar{This is already mentioned in the prev subsection.}In many examples, the case where $\bs w(\cdot)\equiv 1$ is used.
}An example where the decay rate of $\omid{1/ \card{N_n(\bs o)}}$ is strictly smaller
than the growth rate of $\omid{\card{N_n(\bs o)}}$ can be found in~\cite{III}.
\del{\mar{I think we should either delete this or ask this question: Is there an ergodic example?}However, this example is not ergodic (see Remark~I.\ref{I-rem:nonergodic}).}

\begin{proposition}
\label{prop:upperbound-Rd}
If $\Phi$ is a point-stationary point process in $\mathbb R^k$ and the metric on $\Phi$ is equivalent 
to the Euclidean metric, then $\dimH{\Phi}\leq k$.
\end{proposition}
\begin{proof}
One may assume the metric on $\Phi$ is the $l_{\infty}$ metric without loss of generality.
Let $C:=[0,1)^k$ and $\bs U$ be a random point in $-C$ chosen uniformly. For all discrete subsets
$\varphi\subseteq\mathbb R^k$ and $v\in \varphi$, let $\bs C(v)$ be the cube containing $v$ of the
form $C+\bs U+z$ (for $z\in \mathbb Z^k$) and $\bs w_{\varphi}(v):= 1/\card{(\varphi\cap \bs C(v))}$.
Now, $\bs w$ is an equivariant weight function. The construction readily implies that $\bs w(N_r(\bs o))\leq (2r+1)^k$.
Moreover, by $\bs w\leq 1$, one has $\omid{\bs w(0)}<\infty$. Therefore, the unimodular Billingsley
lemma (Theorem~\ref{thm:billingsley}) implies that $\dimH{\Phi}\leq k$.
\end{proof}

\begin{proposition}
\label{prop:palm}
If $\Psi$ is a stationary point process in $\mathbb R^k$ {with finite intensity} and $\Psi_0$
is its Palm version, then
$\dimM{\Psi_0}=\dimH{\Psi_0}=k.$
Moreover, {the modified unimodular Hausdorff measure of $\Psi_0$ satisfies}
	$
	\mathcal M'_k(\Psi_0)= {2^k}{\rho(\Psi)},
	$
where $\rho(\Psi)$ is the intensity of $\Psi$.
\end{proposition}
Notice that if $\Psi_0\subseteq \mathbb Z^k$, {then the claim is directly implied by Theorem~I.3.34.} \del{\mar{$\Psi_0$ is not stationary!}\fra{is stationary, the same 
property holds. This follows from} Theorem~I.3.34.
}The general case is treated below. 
\begin{proof}
For the first claim, by Proposition~\ref{prop:upperbound-Rd}, it is enough to prove that $\dimMl{\Phi}\geq k$. 
Let $\Psi'$ be a shifted square lattice independent of $\Psi$ (i.e., $\Psi'=\mathbb Z^k+\bs U$, 
where $\bs U\in [0,1)^k$ is chosen uniformly, independently of $\Psi$).
Let $\Psi'':=\Psi\cup \Psi'$. Since $\Psi''$ is a superposition of two independent stationary point processes,
it is a stationary point process itself\del{ (see e.g., \cite{bookDaVe03I})}.
By letting $p:=\rho(\Psi)/(\rho(\Psi)+1)$, the Palm version $\Phi''$ of $\Psi''$ is obtained 
by the superposition of $\Phi$ and an independent stationary lattice with probability $p$ (heads),
and the superposition of $\mathbb Z^k$ and $\Psi$ with probability $1-p$ (tails). {So Lemma~\ref{lem:lowerboundiid} implies that $\dimMl{\Phi''}\geq k$.}
\del{For all $n\in \mathbb N$, consider the disjoint $n$-covering of $\mathbb Z^k$ in Example~I.3.12.
In both cases above (heads or tails), one can consider this covering as a random subset of the shifted lattice.
It is easy to see that it provides an equivariant $(n+1)$-covering of $\Phi''$
(note that by enlarging the balls, all of $\mathbb R^k$ is covered).
Also, the probability of having a ball centered at the \rooot{} is $(1-p) (2n+1)^{-k}$. It follows that $\dimMl{\Phi''}\geq k$. 
}
Note that $\Phi''$ has two natural equivariant subsets which, after conditioning to contain the origin,
have the same distributions as $\Phi$ and $\mathbb Z^k$ respectively.
Therefore, one can use Theorem~I.3.34 to deduce that $\dimMl{\Phi}\geq \dimMl{\Phi''}=k$.
Therefore, Proposition~\ref{prop:upperbound-Rd} implies that $\dimH{\Phi} = \dimM{\Phi}=k$.
	
Also, by using Theorem~I.3.34 twice, one gets $\mathcal M'_k(\Phi)= p \mathcal M'_k(\Phi'')$
and $\mathcal M'_k(\mathbb Z^k) =  (1-p)\mathcal M'_k(\Phi'')$. Therefore, 
	$
		\mathcal M'_k(\Phi) = p/{(1-p)} \mathcal M'_k(\mathbb Z^k). 
	$
By the definition of $\mathcal M'_k$, one can directly show that $\mathcal M'_k(\mathbb Z^k)=2^k$ (see also Proposition~I.3.29). This implies the claim.
\end{proof}

The last claim of Proposition~\ref{prop:palm} suggests the following, which is verified when $k=1$ in the next proposition.

\begin{conjecture}
\label{conj:point-stationary}
If $\Phi$ is a point-stationary point process in $\mathbb R^k$ which is not the Palm version of
any stationary point process, then $\measH{k}(\Phi)=0$. 
\end{conjecture}

\begin{proposition}
\label{prop:conj:point-stationary}
Conjecture~\ref{conj:point-stationary} is true when $k=1$.
\end{proposition}
\begin{proof}
Denote $\Phi$ as $\Phi=\{S_n: n\in\mathbb Z\}$ such that $S_0=0$ and $S_n<S_{n+1}$ for each $n$.
Then, the sequence {$T_n:=S_{n+1}-S_n$} is stationary under shifting the indices (see Example~I.2.6). 
The assumption that $\Phi$ is not the Palm version of a stationary point process is equivalent
to $\omid{S_1}=\infty$ (see~\cite{bookDaVe03II} or Proposition~6 of~\cite{shift-coupling}).
Indeed, if $\omid{S_1}<\infty$, then one could bias the probability measure by $S_1$ 
(Definition~I.C.1) 
and then shift the whole process by $-\bs U$, where $\bs U\in [0,S_1]$
is chosen uniformly and independently.
	
Since $\omid{S_1}=\infty$, Birkhoff's pointwise ergodic theorem~\cite{bookPe89} implies
that {$\lim_n (T_1+\cdots+T_n)/n=\infty$}. This in turn implies that $\lim_r \card{N_r(0)}/r = 0$.
Therefore, Lemma~\ref{lem:MDPstronger2} gives that $\contentH{1}{\infty}(\Phi)=\infty$; i.e., $\measH{1}(\Phi)=0$.
\end{proof}

\subsection{Connections to Birkhoff's Pointwise Ergodic Theorem}
\del{This subsection discusses a corollary of the unimodular Billingsley lemma. The reader may skip it at first reading.}

The following corollary of the unimodular Billingsley lemma is of independent interest.
Note that the statement does not involve dimension. 

\begin{theorem}
\label{thm:birkhoff}
Let $[\bs D, \bs o]$ be a unimodular discrete space. 
For any two equivariant weight functions $\bs w_1$ and $\bs w_2$, 
if $\omid{\bs w_1(\bs o)}<\infty$ and  $\bs w_2(\cdot)$ is non-degenerate a.s., then
	\[
	\growthl{\bs w_1(N_r(\bs o))}\leq \growthu{\bs w_2(N_r(\bs o))}, \quad a.s.
	\]
In particular, if $\bs w_1(N_r(\bs o))$ and $\bs w_2(N_r(\bs o))$ have well defined growth rates,
then their growth rates are equal.
\end{theorem}

Note that the condition $\omid{\bs w_1(\bs o)}<\infty$ is necessary as shown in Remark~\ref{rem:bil-nec}.

\begin{proof}
Let $\epsilon>0$ be arbitrary and
$$A:=\{[D,o]\in\dstar: \growthl{\bs w_1(N_r(o))}> \growthu{\bs w_2(N_r(o))}+\epsilon\}.$$ 
It can be seen that $A$ is a measurable subset of $\dstar$. Assume $\myprob{[\bs D,\bs o]\in A}>0$. 
Denote by $[\bs D',\bs o']$ the random \rooted{} discrete space obtained by conditioning $[\bs D,\bs o]$ on $A$.
Since $A$ does not depend on the root (i.e., if $[D,o]\in A$, then $\forall v\in D: [D,v]\in A$),
by a direct verification of the mass transport principle~(I.2.1),
one can show that $[\bs D',\bs o']$ is unimodular.
So by using the unimodular Billingsley
lemma (Theorem~\ref{thm:billingsley}) twice, one gets
\[
\essinf \left(\growthl{\bs w_1(N_r(\bs o'))}\right)\leq \dimH{\bs D'}\leq \essinf \left(\growthu{\bs w_2(N_r(\bs o'))}\right).
\]
By the definition of $A$, this contradicts the fact that $[\bs D',\bs o']\in A$ a.s.
So $\myprob{[\bs D,\bs o]\in A}=0$ and the claim is proved.
\end{proof}

\begin{remark}
Theorem~\ref{thm:birkhoff} is a generalization of a weaker form of Birkhoff's pointwise ergodic
theorem as explained below. In the cases where $\bs D$ is either $\mathbb Z$, the Palm version
of a stationary point process in $\mathbb R^k$ or a point-stationary point process in $\mathbb R$,
Birkhoff's pointwise ergodic theorem (or its generalizations) implies that
$\lim {\bs w_1(N_r(\bs o))}/{\bs w_2(N_r(\bs o))} = \omid{\bs w_1(0)}/\omid{\bs w_2(0)}$ a.s.
This is stronger than the claim of Theorem~\ref{thm:birkhoff}.
Note that Theorem~\ref{thm:birkhoff} implies nothing about 
$\lim {\bs w_1(N_r(\bs o))}/{\bs w_2(N_r(\bs o))}.$ 
On the other side, note that \textit{amenability} is not assumed in this Theorem,
which is a general requirement in the study of ergodic theorems.
However, it will be proved in~\cite{I} that, roughly speaking, non-amenability implies 
$\growthu{\bs w_2(N_r(\bs o))}=\infty$, which makes the claim of Theorem~\ref{thm:birkhoff} trivial
in this case. 
In this case, using exponential gauge functions seems more interesting.
	
\todel{
\begin{enumerate}[(i)]
\item \label{rem:birkhoff:Z} If $\bs D=\mathbb Z$, then $\bs w_1$ and $\bs w_2$ are stationary
under the shift $n\mapsto n-1$ (see Example~I.2.6).
Therefore, Birkhoff's pointwise ergodic theorem (see e.g., \cite{bookPe89})
implies that $\lim {\bs w_1(N_r(\bs o))}/{\bs w_2(N_r(\bs o))} = \omid{\bs w_1(0)}/\omid{\bs w_2(0)}$ a.s.
\fra{In this case, the last relation is stronger than the claim of
Theorem~\ref{thm:birkhoff} in that this theorem only gives that if 
$w_1$ and $w_2$ have well defined growth rates, then these rates coincide.}

\item If $[\bs D, \bs o]$ is a point-stationary point process in $\mathbb R$,
then the above argument still holds by using stationarity under shifting 
the origin to its next point (see Example~I.2.6).
			
\item If $[\bs D, \bs o]$ is the Palm version of a stationary point process, 
the cross-ergodic theorem (see e.g., \cite{bookBaBr94}) also gives the property stated in~\eqref{rem:birkhoff:Z}.

\invisible{\marginpar{\fra{What conclusion exactly}}
\item If $[\bs D, \bs o]$ is an \textit{equivariantly amenable} unimodular graph; e.g., the Cayley graph of an amenable group, then the conclusion of the first part holds. See Section~8 of~\cite{processes} (amenability of unimodular discrete spaces is discussed in~\cite{I}). 
\mar{later: modify after finalizing contents of~\cite{I}.} {\textbf{Update:} There are amenable spaces where the balls are not Folner sets. So I think the argument is wrong in these cases.}
\item (Later: Write actions of amenable groups? It might need new discussions on how to create a unimodular discrete space).
			}

\end{enumerate}

\invisible{{\textbf{Update:} We conjecture that in the non-amenable case, $\dimH{\bs D}=\infty$. If the conjecture is true and the amenable case is proved (see the above update), then all cases are implied by Birkhoff's pointwise ergodic theorems.}}
}
\end{remark}

{
\begin{problem}
	\label{prob:growth}
	Is it true that for every unimodular discrete space $[\bs D, \bs o]$, the growth rates 
	$\growthu{\bs w(N_r(\bs o))}$ and $\growthl{\bs w(N_r(\bs o))}$ do not depend on  $\bs w$ as long as $0<\omid{\bs w(\bs o)}<\infty$?
\end{problem}
}

\subsection{Notes and Bibliographical Comments}

As already mentioned, the unimodular mass distribution principle and the unimodular Billingsley lemma have analogues in the continuum setting (see e.g., \cite{bookBiPe17}) and are named accordingly.  
Note however that there is no direct or systematic reduction to these continuum results. 
For instance, in the continuum setting, one should assume that the space under study is a subset of the Euclidean space, or more generally, satisfies the \textit{bounded subcover property} ({see e.g.,~\cite{bookBiPe17}}). Theorem~\ref{thm:billingsley} does not require such assumptions.
Note also that the term $\growthu{\bs w(N_r(\bs o))}$ in Theorem~\ref{thm:billingsley}
does not depend on the \rooot{} in contrast to the analogous term in the continuum version.
Similar observations can be made on Theorem \ref{thm:mdp-simple}.

\section{{Examples}}
\label{sec:examples}

{This section presents some examples for illustrating the results of the previous section. It also provides further results on the examples introduced in Section~I.4.}

\del{The structure of this section is analogous to that of Section~I.4.}

\subsection{General Unimodular Trees}
\label{subsec:trees}
{The following is a direct corollary of Theorem~I.4.2 and the unimodular Billingsley lemma. Since the statement does not involve dimension, it is of independent interest and believed to be new.

\begin{corollary}
	For every unimodular one-ended tree $[\bs T, \bs o]$ and every equivariant weight function $\bs w$, almost surely,
	\[
		\decayu{\myprob{h(\bs o)=n}} \leq \growthu{\bs w(N_r(\bs o))} \leq \growthu{\omid{\bs w(N_r(\bs o))}}.
	\]
\end{corollary}

The rest of this subsection is focused on unimodular trees with infinitely many ends.}

\begin{proposition}
	\label{prop:tree-expGrowth}
	Let $[\bs T, \bs o]$ be a unimodular tree with infinitely many ends\del{ a.s.}
	such that $\omid{\mathrm{deg}(\bs o)}<\infty$. Then $\bs T$ has exponential growth a.s. {and $\dimH{\bs T}=\infty$}.
\end{proposition}

{In fact, the assumption $\omid{\mathrm{deg}(\bs o)}<\infty$ is not necessary. Also, the graph-distance metric on $\bs T$ can be replaced by an arbitrary equivariant metric. These will be proved in~\cite{I}.
}

The following proof uses the definitions and results of~\cite{processes}, but {they} \del{the definitions }are not recalled for brevity.
\begin{proof}[Proof of Proposition~\ref{prop:tree-expGrowth}]
	By Corollary~8.10 of~\cite{processes}, $[\bs T, \bs o]$ is \textit{non-amenable} (this will be discussed further in~\cite{I}). 
	So Theorem~8.9 of~\cite{processes} implies that the critical probability $p_c$ of percolation on $\bs T$ is
	less than one {with positive probability. In fact, it can be shown that $p_c<1$ a.s. (if not, condition on the event $p_c=1$ to get a contradiction).}
	For any tree, $p_c$ is equal to the inverse of the \textit{branching number}.
	So the branching number is more than one, which implies that the tree has exponential growth. {Finally, the unimodular Billingsley lemma (Theorem~\ref{thm:billingsley}) implies that $\dimH{\bs T}=\infty$.}
\end{proof}

{
	The following example shows that the Minkowski dimension can be finite.
	\begin{example}
		Let $T$ be the 3-regular tree. Split each edge $e$ by adding a random number $\bs l_e$ of new vertices and let $\bs T_0$ be the resulting tree. Let $v_e$ be the middle vertex in this edge (assuming $\bs l_e$ is always odd) and assign marks by $\bs m_0(v_e):=\bs l_e$. 
		Assume that the random variables $\bs l_e$ are i.i.d. If $\omid{\bs l_e}<\infty$, then one can bias the probability measure and choose a new root to obtain a unimodular marked tree, namely $[\bs T,\bs o;\bs m]$ (see Example~9.8 of~\cite{processes} or~\cite{shift-coupling}). 
		It will be shown below that $\dimMu{\bs T}$ may be finite. 
		\\
		Let $\bs R$ be an arbitrary equivariant $r$-covering of $\bs T$. Consider the set of middle vertices $\bs A_r:=\{v\in\bs T: \bs m(v)\geq r\}$. Since these vertices have pairwise distance at least $r$, they belong to different balls in the covering. So, by the mass transport principle, one can show that $\rho(\bs R)\geq \rho(\bs A_r)$, where $\rho(\cdot)=\myprob{\bs o\in \cdot}$ denotes the intensity. On the other hand, let $\bs S$ be the equivariant subset of vertices with degree 3. Send unit mass from every point of $\bs A_r$ to its two closest points in $\bs S$. Then the mass transport principle implies that $2\rho(\bs A_r)=3\rho(\bs S)\myprob{\bs l_e\geq r}$. Hence, $\rho(\bs R)\geq \frac 3 2 \rho(\bs S)\myprob{\bs l_e\geq r}$. This gives that $\dimMu{\bs T}\leq \decayu{\myprob{\bs l_e\geq r}}$, which can be finite. In fact, if $\decay{\myprob{\bs l_e\geq r}}$ exists, Proposition~\ref{prop:infEndslowerbound}	 below implies that $\dimM{\bs T}=\decay{\myprob{\bs l_e\geq r}}$.
	\end{example}

The following proposition gives a lower bound on the Minkowski dimension. 

\begin{proposition}
	\label{prop:infEndslowerbound}
	Let $[\bs T, \bs o]$ be a unimodular tree with infinitely many ends and without leaves. Let $\bs S$ be the equivariant subset of vertices of degree at least 3. For every $v\in\bs S$, let $\bs w(v)$ be the sum of the distances of $v$ to its neighbors in $\bs S$. If $\omid{\bs w(\bs o)^{\alpha}}<\infty$, then $\dimMl{\bs T}\geq \alpha$.
\end{proposition}

The proof is based on the following simpler result. This will be used in Subsection~\ref{subsec:pwit} as well.

\begin{proposition}
	\label{prop:regree-distorted}
	Let $[\bs T,\bs o]$ be a unimodular tree such that the degree of every vertex is at least 3. Let $\bs d'$ be an equivariant metric on $\bs T$. Let $\bs w(v):=\sum_u \bs d'(v,u)$, where the sum is over the 3 neighbors of $v$ which are closest to $v$ under the metric $\bs d'$. If $\omid{\bs w(\bs o)^{\alpha}}<\infty$, then $\dimMl{\bs T,\bs d'}\geq\alpha$.
\end{proposition}
\begin{proof}
	Define $\bs w'(v):=\sum_{u} \bs d'(u,v)^{\alpha}$, where the sum is over the three closest neighbors of $v$.
	It is enough to assume that $\bs d'$ is generated by equivariant edge lengths since increasing the edge lengths does not increase the dimension (by Theorem~I.3.31). By the same argument, it is enough to assume $\bs d'(u,v)\geq 1$ for all $u\sim v$. Then, 
	it can be seen that there exists a constant $c$, that depends only on $\alpha$, such that $\bs w'(\nei{r}{v})\geq c r^{\alpha}$ for all $v\in\bs T$ and $r\geq 0$ (this is implied by Lemma~\ref{lem:regtree-weight}). 
	Also, the assumption implies that $\omid{\bs w'(\bs o)}<\infty$.
	So Lemma~\ref{lem:lowerboundiid} implies that $\dimMl{T_3,\bs d'}\geq \alpha$ and the claim is proved.
\end{proof}

\begin{proof}[Proof of Proposition~\ref{prop:infEndslowerbound}]
	For $v\in \bs S$, let $\bs w'(v):=\sum_u d(u,v)^{\alpha}$, where the sum is over the neighbors of $v$ in $\bs S$. For $v\in\bs T\setminus\bs S$, if $u_1$ and $u_2$ are the two closest points of $\bs S$ to $v$, let $g(v,u_i):=d(u_i,v)^{\alpha-1}$ and $\bs w'(v):=g(v,u_1)+g(v,u_2)$.
	The assumption implies that $\omid{\bs w'(\bs o)}<\infty$ (use the mass transport principle for $g$ defined above).
	Similarly to Proposition~\ref{prop:regree-distorted}, there exists $c=c(\alpha)$, such that $\bs w'(\nei{r}{v})\geq c r^{\alpha}$ for all $v\in\bs T$ and $r\geq 0$ (this is implied by Lemma~\ref{lem:regtree-weight}) and the claim is proved.
%
\end{proof}
}

\del{
========================================

{\textbf{Update:} The proof of Thm \ref{thm:tree-expGrowth} is wrong. We should use ideas pertaining to amenability, which are discussed in Part III. So I suggest to move it to Part III and only mention it here. I suggest to keep here only these: Lemma~\ref{lem:tree-expGrowth}, Problem~\ref{prob:tree-infEndMinkowski}, and a shorter version of the claims (without proof) of Prop\ref{prop:regree-distorted} and Lemma~\ref{lem:regtree-generalization} (and with reference to the arXiv version). The latter will be used in PWIT, but I suggest to leave it to the reader as well. Do you agree? We can discuss more.}
\del{
\subsubsection{Unimodular Two-Ended Trees}
\label{subsec:two-ended}
\fra{The unimodular Billingsley lemma can be used to give
another proof of the property that the Minkowski and Hausdorff 
dimensions of any unimodular
two-ended tree are equal to one (Theorem~I.4.1).}

Let $[\bs T, \bs o]$ be a unimodular two-ended tree.
For all two-ended trees $T$ and $v\in T$, let $\bs w_T(v)$ be 1 if $v$
belongs to the trunk of $T$ and 0 otherwise.
It can be seen that $w$ is an equivariant process
(Definition~I.2.8). 
Let $c(v)$ be the distance of $v$ to the trunk of $T$.
For $n\in\mathbb N$ larger than $c(v)$, one has \fra{$\bs w(N_n(v))=2(n-c(v))+1$.}
Therefore, the unimodular Billingsley lemma (Theorem~\ref{thm:billingsley}) implies that $\dimH{\bs T}\leq 1$. 
On the other hand, Corollary~\ref{cor:graphs-lowerbound} implies that  $\dimMl{\bs T}\geq 1$.
Therefore, $\dimM{\bs T}=\dimH{\bs T}=1$.

\subsubsection{Unimodular Trees with Infinitely Many Ends}
}
\label{sec:trees-infEnds}

\begin{lemma}
\label{lem:tree-expGrowth}
Let $[\bs T, \bs o]$ be a unimodular tree with infinitely many ends a.s.
such that $\omid{\mathrm{deg}(\bs o)}<\infty$. Then $\bs T$ has exponential growth a.s. 
\end{lemma}
The proof uses the definitions and results of~\cite{processes}, but the definitions are not recalled for brevity.
\begin{proof}
By Corollary~8.10 of~\cite{processes}, $[\bs T, \bs o]$ is \textit{non-amenable} (this will be discussed further in~\cite{I}). 
So Theorem~8.9 of~\cite{processes} implies that the critical probability $p_c$ of percolation on $\bs T$ is
less than one \forlater{with positive probability. In fact, it can be shown that $p_c<1$ a.s. (if not, condition on the event $p_c=1$ to get a contradiction).}
For any tree, $p_c$ is equal to the inverse of the \textit{branching number}.
So the branching number is more than one, which implies that the tree has exponential growth.
\end{proof}

The following theorem removes the assumption $\omid{\mathrm{deg}(\bs o)}<\infty$.

\begin{theorem}
\label{thm:tree-expGrowth}
If $[\bs T, \bs o]$ is a unimodular tree with infinitely many ends a.s.,
then $\bs T$ has exponential growth a.s. and  $\dimH{\bs T}=\infty$.
\end{theorem}
\begin{proof}
By the unimodular Billingsley lemma (Theorem~\ref{thm:billingsley}), it is enough to prove the first claim.
	
Split all edges in $\bs T$ into 3 equidistant parts by adding two new vertices on each edge. 
Let $\bs T_1$ be the resulting tree. Then, $[\bs T_1, \bs o]$ is not necessary unimodular, but,
by a suitable biasing 
(Definition~I.\ref{I-def:bias}) 
and changing the \rooot{}, one can obtain a unimodular tree
(see Example~9.8 of~\cite{processes} for the precise formula).
It can be seen that it is enough to prove the claim for the new unimodular tree.
	
Hence, one can assume that every two vertices of $\bs T$ with degree at least
3 have distance at least 3 without loss of generality. 
For every vertex $v$ with degree at least 3, do the following: Let $v_1,\ldots, v_k$ be the neighbors of $v$. 
Let $\pi$ be a random permutation of $\{1,\cdots,k\}$ chosen uniformly.
Then delete the edge $vv_{\pi_i}$ and add an edge between $v_{\pi_i}$ and $v_{\pi_{i+1}}$ 
for every $1\leq i\leq k-1$ ($v$ is still adjacent with $v_{\pi_k}$). 
Do this for all vertices $v$ independently. One can see that this construction fits
in the context of equivariant processes (see Remark~I.2.12).
It follows that if $\bs T'$ denotes the resulting tree, then $[\bs T', \bs o]$ is unimodular.
Also, $\bs T'$ has infinitely many ends and the degree of each vertex is at most 3
(which is implied by the above assumption). So, Lemma~\ref{lem:tree-expGrowth} implies
that $\bs T'$ has exponential growth a.s. Since the graph-distance metric of $\bs T'$ is
at least half of the graph-distance metric of $\bs T$, it follows that $\bs T$
has also exponential growth. So the claim is proved.
\end{proof}

%
%
%

\forlater{
\begin{remark}
	\label{rem:tree-nonamenable}
	The claim of Theorem~\ref{thm:tree-expGrowth} still holds if the graph-distance metric on $\bs T$ is replaced by any equivariant metric. This follows from the results of~\cite{I} on \textit{non-amenable} unimodular discrete spaces.
\end{remark}
}

%


The following results study lower bounds on the Minkowski dimension in the setting of Problem~\ref{prob:tree-infEndMinkowski}.
Recall from Subsection~I.\ref{I-subsec:inf-ended} that an equivariant metric $\bs d'$ is said to be
\defstyle{generated by equivariant edge lengths} (or a \textit{geodesic metric}) if for every path
$v_1v_2\ldots v_k$, one has $\bs d'(v_1,v_k) = \sum_{i=1}^{k-1}\bs d'(v_i,v_{i+1})$.
In the following, the ball of radius $r$ under the metric $\bs d'$ and centered at $v$ is denoted by $N'_r(v)$.

\begin{proposition}
	\label{prop:regree-distorted}
	Let $\bs d'$ be an equivariant metric on the 3-regular tree $T_3$ which is generated by {equivariant edge lengths}. If the random variable $\sum_{v\sim \bs o} \bs d'(\bs o,v)$ has finite moment of order $\alpha$, then $\dimMl{T_3,\bs d'}\geq \alpha$. In particular, if it has finite moments of any order, then $\dimM{T_3,\bs d'} =\dimH{T_3,\bs d'}=\infty$.
\end{proposition}

\begin{proof}
It is enough to assume $\bs d'(x,y)\geq 1$ for all $x\sim y$ since increasing the edge lengths does
not increase the dimension (Theorem~I.3.31).
Consider the following equivariant weight function on $T_3$:
\begin{equation}
\label{eq:thm:regtree-dimension:w}
\bs w(u):=C\sum_{v \sim u} \bs d'(u,v)^{\alpha},
\end{equation}
where $C$ is a constant such that
\begin{equation}
\label{eq:thm:regtree-dimension:C}
\forall x\in [0,1]: Cx^{\alpha} + (1-x)^{\alpha}\geq \frac 12.
\end{equation}
It is easy to see that such a $C$ exists. Now, Lemma~\ref{lem:regtree-weight} below,
which is a deterministic result, implies that $\bs w(N'_r(\bs o))\geq r^{\alpha}$ a.s. for every $r\geq 0$.
Also, the assumption on $\bs d'$ implies that $\bs w(\bs o)$ has finite mean.
So Lemma~\ref{lem:lowerboundiid} implies that $\dimMl{T_3,\bs d'}\geq \alpha$ and the claim is proved.
\end{proof}

The following lemma is used in the last proof.

\begin{lemma}
\label{lem:regtree-weight}
Let $\alpha<\infty$ and $(T,o)$ be a deterministic rooted tree such that
$\mathrm{deg}(o)\geq 2$ and $\mathrm{deg}(v)\geq 3$ for all $v\neq o$.
Let $d'$ be a metric on $T$ which is generated by a function on the edges such that $d'(\cdot)\geq 1$.
Define the weight function $w$ on $T$ and the constant $C$
by~\eqref{eq:thm:regtree-dimension:w} and~\eqref{eq:thm:regtree-dimension:C} respectively.
Then, for all $r\geq 0$, one has $w(N'_r(o))\geq r^{\alpha}$.
\end{lemma}

\begin{proof}
For $r\geq 0$, let $f(r)$ be the infimum value of $w(N'_r(o))$ for all trees
with the stated conditions. So one should prove $f(r)\geq r^{\alpha}$.
The claim is true for $r=0$. Also, if $0<r<1$, one has $N'_r(o)=\{o\}$ and the claim is trivial. The proof uses
induction on $\lfloor r\rfloor$. Assume that $r\geq 1$ and for all $s<\lfloor r\rfloor$, one has $f(s)\geq s^{\alpha}$.
For $y\sim o$, let $T_y$ be the connected component containing $y$ when the edge $(o,y)$ is removed.
It can be seen that $[T_y,y]$ satisfies the conditions of the lemma. Therefore, one obtains 
\begin{eqnarray*}
w(N'_r(o)) &=&  w(o) + \sum_{y: y\sim o}  w(N'_{r- d'(o,y)}(T_y, y))\\
&\geq &	 w(o) + \sum_{y: y\sim o} f(r- d'(o,y))\\
&\geq & \sum_{y: y\sim o} \left[C  d'(o,y)^{\alpha} + (r- d'(o,y))^{\alpha}\right]\\
&\geq & \mathrm{deg}(o)\cdot \min_{0\leq x \leq r}\{Cx^{\alpha} +  (r-x)^{\alpha}\}\\
&\geq & \mathrm{deg}(o) r^{\alpha}/2\\
&\geq & r^{\alpha},
\end{eqnarray*}
where the third line is by the definition of $ w(o)$ in~\eqref{eq:thm:regtree-dimension:w}
and the induction hypothesis, the fifth line is due to~\eqref{eq:thm:regtree-dimension:C},
and the last line is by the assumption $\mathrm{deg}(o)\geq 2$.
This implies that $f(r)\geq r^{\alpha}$ and the induction claim is proved.
\end{proof}

The following is a slight generalization of Proposition~\ref{prop:regree-distorted}, which will be used in Subsection~\ref{subsec:pwit}.

\begin{lemma}
\label{lem:regtree-generalization}
Let $[\bs T, \bs o]$ be a unimodular tree  in which the degree of every vertex is at least 3 a.s. 
Let $\bs d'$ be an equivariant metric on $[\bs T, \bs o]$ which is generated by equivariant edge lengths. 
For $u\in \bs T$, let $\bs w(u)$ be the third minimum number in the multi-set $\{\bs d'(u,v)^{\alpha}:v\sim u \}$.
If $\omid{\bs w(\bs o)}<\infty$, then $\dimMl{{\bs T, \bs d'}}\geq \alpha$.
\end{lemma}

\begin{proof}
By Lemma~\ref{lem:lowerboundiid}, it is enough to show that $\forall r\geq 0: \bs w(N'_r(\bs o))\geq \frac 1 C r^{\alpha}$,
where $C$ is defined by~\eqref{eq:thm:regtree-dimension:C}. Notice that this is a deterministic claim.
So consider a realization $(T,o)$ of $[\bs T, \bs o]$. Define the subtree $(T',o)$ of $T$ by adding vertices
recursively as follows. First, add 3 neighbors of $o$ which are closest to $o$ (under the metric $\bs d'$) to $T'$.
Then, recursively, for every newly added vertex $u$, add 2 neighbors of $u$ which are closest to $u$ and not already added.
It is straightforward that $T'$ is a 3-regular tree and $\bs w(u)\geq \sum_v \bs d'(u,v)^{\alpha}$,
where the sum is over the neighbors of $u$ in $T'$. So Lemma~\ref{lem:regtree-weight} implies that
$\bs w(N'_r(T',o))\geq \frac 1 C r^{\alpha}$. Hence, $\bs w(N'_r(T,o))\geq \frac 1 C r^{\alpha}$. So the claim is proved.
\end{proof}
}


\subsection{Instances of Unimodular Trees}

\subsubsection{A Unimodular Tree With No Growth Rate}
\label{subsec:canopy-generalized}
Recall the generalized canopy tree $[\bs T, \bs o]$ from Subsection~I.4.2.2. Here, it is shown that $\growthl{\bs T}\neq \growthu{\bs T}$ if the parameters are suitably chosen. Similarly, it provides an example where the exponential growth rate does not exist. In the latter, the existence of such trees is proved in~\cite{Ti14}, but with a more difficult construction.

Choose the sequence $(p_n)_n$ in the definition of $[\bs T, \bs o]$ such that $p_n=c2^{-q_n}$ and $\sum_n p_n=1$, where $c$ is constant and $q_0\leq q_1\leq \cdots$ is a sequence of integers. In this case, $\bs T$ is obtained by splitting the edges of the canopy tree by adding new vertices or concatenating them, depending only on the \textit{level} of the edges. It can be seen that if $v$ is a vertex in the $n$-th level of $\bs T$, then the number of descendants of $v$ is $(p_0+\cdots+p_n)/p_n$. It follows that $\growthl{\bs T}=\decayl{p_n}$ and $\growthu{\bs T}=\decayu{p_n}$. So, by choosing $(p_n)_n$ appropriately, $\bs T$ can have no polynomial (or exponential) growth rate. This proves the claim. 
Note also that the unimodular Billingsley lemma and  Theorem~I.4.2 imply that $\dimH{\bs T}=\growthu{\bs T}$ here.

\subsubsection{Unimodular Galton-Watson Trees}
\label{subsec:ugw}

Here, it is shown that the unimodular Galton-Watson tree~\cite{processes} is infinite dimensional. 
(note that this tree differs from the Eternal Galton-Watson tree of Subsection~I.4.2.3 which is a directed tree). 
Consider an ordinary Galton-Watson tree with offspring distribution $\mu=(p_0,p_1,\ldots)$, where $\mu$ is a probability measure on $\mathbb Z^{\geq 0}$.
The {unimodular Galton-Watson tree} $[\bs T, \bs o]$ has a similar construction with the
difference that the offspring distribution of the \rooot{} is different from that of the other vertices:
It has for distribution {the size-biased version} $\hat{\mu}=(\frac n m p_n)_n$, where $m$ is the mean of $\mu$ (assumed to be finite).

In what follows, the trivial case $p_1=1$ is excluded.
If $m\leq 1$, then $\bs T$ is finite a.s.; i.e., there is extinction a.s.
Therefore, \del{Proposition~I.\ref{I-prop:finite-HausDim} implies that }$\dimH{\bs T}=0$.
So assume the \textit{supercritical case}, namely $m>1$.
If $p_0>0$, then $\bs T$ is finite with positive probability.
So $\dimH{\bs T}=0$ for the same reason. Nevertheless, one can condition on non-extinction as follows.

\begin{proposition}
Let $[\bs T, \bs o]$ be a supercritical unimodular Galton-Watson tree conditioned on non-extinction. Then,
$
\dimM{\bs T}=\dimH{\bs T}=\infty.
$
\end{proposition}
\begin{proof}
The result for the Hausdorff dimension is followed from the unimodular Billingsley lemma
(Theorem~\ref{thm:billingsley}) and the Kesten-Stigum theorem~\cite{KeSt66},
which implies that $\lim_n \card{N_n(\bs o)}m^{-n}$ exists and is positive a.s. 
Computing the Minkowski dimension is more difficult. By part~\eqref{part:lem:lowerbound:exp} of Lemma~\ref{lem:lowerboundiid},
it is enough to prove that $\omid{(1-n^{-\alpha})^{\card{\nei{n}{\bs o}}}}$ has infinite decay rate for every $\alpha\geq 0$.
Denote by $[\widetilde{\bs T},\widetilde{\bs o}]$ the Galton-Watson tree with the same parameters.
Using the fact that $\card{N_n(\bs o)}$ is stochastically larger than $\card{N_{n-1}(\widetilde{\bs o})}$,
one gets that it is enough to prove the last claim for $[\widetilde{\bs T},\widetilde{\bs o}]$. 
	
For simplicity, the proof is given for the case $p_0=0$ only. By this assumption, the probability of extinction is zero. The general case can be proved with 
similar arguments and by using the decomposition theorem of supercritical Galton-Watson trees
(see e.g.,  Theorem~5.28 of~\cite{bookLyPe16}). In fact, the following proof implies the general claim by the fact that the \textit{trunk}, conditioned on non-extinction, is another supercritical unimodular Galton-Watson tree. The latter can be proved similarly to the decomposition theorem.
	
Let $f(s):=\sum_n p_n s^n$ be the generating function of $\mu$.
By classical results of the theory of branching processes, for all $s\leq 1$, 
$\omid{s^{d_n(\widetilde{\bs o})}}=f^{(n)}(s),$
where $d_n(\tilde{\bs o}):=\card{N_n(\widetilde{\bs o})}-\card{N_{n-1}(\widetilde{\bs o})}$
and $f^{(n)}$ is the $n$-fold composition of $f$ with itself. 
Let $a>0$ be fixed and $g(s):=\frac {as}{-s+a+1}$ (such functions are frequently used in
the literature on branching processes; see, e.g., \cite{bookAsHe83}).
One has $f(0)=g(0)=0$, $f(1)=g(1)=1$, $f'(1)=m>1$, $g'(1)=(1+a)/a$, and $f$ is convex.
Therefore, $a$ can be chosen large enough such that $f(s)\leq g(s)$ for all $s\in[0,1]$. So
\begin{eqnarray*}
f^{(n)}(s) \leq g^{(n)}(s) = \frac {a^n s}{a^n + (a+1)^n(1-s)},
\end{eqnarray*}
where the last equality can be checked by induction. Therefore, 
\[
f^{(n)}(1-n^{-\alpha}) \leq \frac {a^n}{a^n + n^{-\alpha}(a+1)^n }.
\]
It follows that $\decay{f^{(n)}(1-n^{-\alpha})}=\infty$. So the above discussion gives
that $\omid{(1-n^{-\alpha})^{\card{\nei{n}{\bs o}}}}$ has infinite decay rate and the claim is proved.
\end{proof}

\invisible{
Note: This may be of independent interest:
It is {well known} that in a supercritical Galton-Watson tree conditioned on non-extinction, the set of non-extincting vertices form another supercritical Galton-Watson tree (later: ref: Branching processes, Athreya, 1972, pp 49, 52). Similarly, one can show that the \textit{trunk} of $\bs T$ (i.e. the set of vertices in which their deletion produces at least two infinite connected components), conditioned on containing the \rooot{}, is by itself another supercritical unimodular Galton-Watson tree. Therefore, by Theorem~I.3.34, it is enough to assume $p_0=0$. \textbf{Update:} Theorem~I.3.34 does not imply this. But the fact that in the trunk, $\omid{(1-n^{-\alpha})^{\card{\nei{n}{\bs o}}}}$ has infinite decay implies it.
}

\invisible{{Idea: What can we say about age-dependent branching processes? See Subsection~\ref{subsec:regtree}}}

\subsubsection{Unimodular Eternal Galton-Watson Trees}
\label{subsec:egw}
Unimodular eternal Galton-Watson (\egw{}) trees were introduced in Subsection~I.4.2.3.
The following theorem complements Theorem~I.4.9.

\begin{proposition}
\label{thm:EGWdimension}
Let $[\bs T, \bs o]$ be a unimodular eternal Galton-Watson tree.
If the offspring distribution has finite variance, then 
$
\dimM{\bs T}= \dimH{\bs T} = 2.
$
\end{proposition}
\begin{proof}
Theorem~I.4.9 proves that $\dimM{\bs T}=2$. So it remains to prove $\dimH{\bs T}\leq 2$.
By the unimodular Billingsley lemma (Theorem~\ref{thm:billingsley}), it is enough to show that 
$\omid{\card{N_n(\bs o)}}\leq cn^2$ for a constant $c$.
Recall from Subsection~I.4.1.2 that $F(v)$ represents the parent of vertex $v$ and $D(v)$ 
denotes the subtree of descendants of $v$. Write $N_n(\bs o)=Y_0\cup Y_1\cup \cdots \cup Y_n$,
where $Y_n:={N_n(\bs o) \cap D(\bs o)}$ and 
$Y_i:={N_n(\bs o) \cap D(F^{n-i}(\bs o))\setminus D(F^{n-i-1}(\bs o))}$ for $0\leq i<n$. 
By the explicit construction of \egw{} trees in~\cite{eft}, 
$Y_n$ is a critical Galton-Watson tree up to generation $n$. Also, for $0\leq i<n$, 
$Y_i$ has the same structure up to generation $i$, except that the distribution of the
first generation is \textit{size-biased minus one} (i.e., $(np_{n+1})_n$ with the notation of Subsection~\ref{subsec:ugw}).
So the assumption of finite variance implies that the first generation in each $Y_i$ has finite mean, namely $m'$.
Now, one can inductively show that $\omid{\card{Y_n}}= n$ and $\omid{\card{Y_i}}=i m'$,
for $0\leq i<n$. It follows that $\omid{\card{N_n(\bs o)}}\leq (1+m')n^2$ and the claim is proved.
\end{proof}

\invisible{\mar{Please see which one of these are suitable to be included.}
{\textbf{Guess:} The exact gauge function for EGW is $r^2\log\log r$, similar to the continuum case in [The	Hausdor measure of stable trees].}

{\textbf{Update:} The idea for proving the stable case (Conjecture~\ref{I-conj:EGW-stable}) is in the picture in the next page. Please see if it can be completed. \textbf{Update:} Maybe a weaker assumption is enough.}
	\textbf{Update:} Lemma 2 in [A branching process with mean one and possibly infinite variance] says that $\myprob{h(\bs o)>n} \sim n^{-\beta}$, but I don't know what is the value of $\beta$ in our language.

\begin{figure}
	\begin{center}
		\includegraphics[width=\textwidth]{../../Attachments/EGWheavytail.jpg}
	\end{center}
\end{figure}
}

\subsubsection{The Poisson Weighted Infinite Tree}
\label{subsec:pwit}
The Poisson Weighted Infinite Tree (\texttt{PWIT}) is defined 
as follows (see e.g., \cite{objective}). It is a rooted tree $[\bs T, \bs o]$ such that the degree
of every vertex is infinite. Regarding $\bs T$ as a family tree with progenitor $\bs o$,
the edge lengths are as follows. For every $u\in \bs T$, the set $\{d(u,v): v\text{ is an offspring of } u\}$
is a Poisson point process on $\mathbb R^{\geq 0}$ with intensity function $x^k$, where $k>0$ is a given integer.
Moreover, for different vertices $u$, the corresponding Poisson point processes are jointly independent.
It is known that the \texttt{PWIT} is unimodular (notice that although each vertex has infinite degree, the \texttt{PWIT} is boundedly-finite as a metric space). See for example~\cite{objective} for more details. 


\begin{proposition}
	The \texttt{PWIT} satisfies
	$
		\dimM{\texttt{PWIT}} = \dimH{\texttt{PWIT}} = \infty.
	$
\end{proposition}

\begin{proof}
Denote the neighbors of $\bs o$ by $v_1,v_2,\ldots$ such that {$d(\bs o,v_i)$ is increasing in $i$.}
It is straightforward that all moments of $d(\bs o,v_3)$ are finite. Therefore,
Proposition~\ref{prop:regree-distorted} implies that $\dimMl{\bs T}=\infty$ (see also Lemma~\ref{lem:regtree-weight}). This proves the claim.
\end{proof}

\subsection{Examples associated with Random Walks}
\label{sstsrw}

As in Subsection~I.4.3, consider the simple random walk $(S_n)_{n\in \mathbb Z}$
in $\mathbb R^k$, where $S_0=0$ and the increments $S_n-S_{n-1}$ are i.i.d. 

\subsubsection{The Image and The Zeros of the Simple Random Walk}
\label{subsec:image}

{Recall that Theorem~I.4.11 studies the unimodular Minkowski dimension of the image of a simple random walk. The following is a complement to this result.}

\begin{theorem}
\label{thm:image}
Let $\Phi:=\{S_n\}_{n\in\mathbb Z}$ be the image of a simple random walk $S$ in $\mathbb R$, where $S_0:=0$.
Assume the jumps $S_n-S_{n-1}$ are positive a.s. 
	
\begin{enumerate}[(i)]
\item \label{part:thm:image:0} $\dimH{\Phi}\leq 1\wedge \decayu{\myprob{S_1>r}}$.

\item \label{part:thm:image:1} If $\beta:=\decay{\myprob{S_1>r}}$ exists, then
$\dimM{\Phi} = \dimH{\Phi}=1\wedge \beta.$
		
\del{
		
\item \label{part:thm:image:3} $\dimH{\Phi}\leq 1\wedge \decayu{\myprob{S_1>r}}$.}
\end{enumerate}
\end{theorem}

\begin{proof}
{Theorem~I.4.11 proves that $\dimMl{\Phi} \geq 1\wedge \decayl{\myprob{S_1>r}}$.}
So it is enough to prove part~\eqref{part:thm:image:0}.
Since $\Phi$ is a point-stationary point process in $\mathbb R$ (see Subsection~I.4.3.1),
Proposition~\ref{prop:upperbound-Rd} implies that $\dimH{\Phi}\leq 1$. Now, assume $\decayu{\myprob{S_1>r}}< \beta$.
Then, there exists $c>0$ such that $\myprob{S_1>r}> c r^{-\beta}$ for all $r\geq 1$.
By using Lemma~\ref{lem:BaumKatz1} twice, for the positive and negative parts of the random walk,
one can prove that there exists $C<\infty$ and a random number $r_0>0$ such that for all $r\geq r_0$,
one has $\card{N_r(\bs o)}\leq Cr^{\beta} \log \log r$ a.s.
Therefore, the unimodular Billingsley lemma (Theorem~\ref{thm:billingsley}) implies that
$\dimH{\Phi}\leq \beta+\epsilon$ for every $\epsilon>0$, which in turn implies that $\dimH{\Phi}\leq \beta$.
So the claim is proved.
\end{proof}

\invisible{I think that the exact value of $\dimH{\bs T}$ can be found by Theorem~4 of [Lower functions for increasing random walks and subordinators]. See Lemma~\ref{lem:BaumKatz1}.}

The following {proposition} complements Theorem~I.4.12. It is readily implied by Theorem~\ref{thm:image} above.
\begin{proposition}
\label{prop:zeros}
Let $\Psi$ be the zero set of the {symmetric} simple random walk on $\mathbb Z$ with uniform jumps in $\{\pm 1\}$. Then,
$ \dimM{\Psi} = \dimH{\Psi}= \frac 12.$
\end{proposition}

	In this proposition, the guess is that $\measH{1/2}(\Psi)=\infty$.
	Additionally, by analogy with the zero set of Brownian motion~\cite{TaWe66}, it is natural to guess that $\sqrt{r\log\log r}$ is a dimension function for $\Psi$ (see Subsection~I.3.8.2). To prove this, one should strengthen Lemma~\ref{lem:BaumKatz1} and also construct a covering of the set which is better than that of Proposition~I.3.14. {For the former, one may use Theorem~4 of~\cite{FrPr71} (it seems that the assumption of~\cite{FrPr71} on the tail of the jumps is not necessary for having an inequality similarly to Lemma~\ref{lem:BaumKatz1}).} For the latter, one might try to get ideas from~\cite{TaWe66} (it is necessary to use intervals with different lengths).

\begin{example}[Infinite Hausdorff Measure]
\label{ex:infiniteMeasure}
In Theorem~\ref{thm:image}, assume  $\myprob{S_1>r}= 1/\log r$ for large enough $r$.
Then, part~\eqref{part:thm:image:1} of the theorem implies that $\dimH{\Phi}=0$.
However, since $\Phi$ is infinite a.s., it has infinite 0-dimensional Hausdorff measure (Proposition~I.3.28).
\end{example}

\begin{example}[Zero Hausdorff Measure]
\label{ex:zeroMeasure}
In Theorem~\ref{thm:image}, assume $\myprob{S_1>r}= 1/ r$ for large enough $r$.
Then, part~\eqref{part:thm:image:1} of the theorem implies that $\dimH{\Phi}=1$.
Since $\omid{S_1}=\infty$, $\Phi$ is not the Palm version of any stationary point process
(see Proposition~\ref{prop:conj:point-stationary}). Therefore, 
Proposition~\ref{prop:conj:point-stationary} implies that $\measH{1}(\Phi)=0$.
\end{example}

\invisible{{
	
	
	\textbf{Guess.} If $\myprob{S_1>r} \geq c r^{-\beta}$ in the theorem, I think $\measH{\beta}(\Phi)=\infty$ (more discussion in the next subsection).
	
	\textbf{New Result:} If $S_1$ is in the domain of attraction of the $\alpha$-stable distribution and satisfies some other conditions, then discrete dimension of $\Phi$ is $\alpha$. So our comparison (Subsection~III.\ref{III-subsec:discreteDim}) implies $\dimH{\Phi}\geq \alpha$. Also, I guess the scaling limit of $\Phi$ is an $\alpha$-stable process (is this true?) (or maybe the $\alpha$-stable regenerative set) whose Hausdorff dimension is $\alpha$ (is this correct?). So our scaling limit result implies $\dimH{\Phi}\leq\alpha$. 
}}

\invisible{
{ 
	
	{\textbf{Question.} Zeros of Brownian motion has dimension $\frac 12$ but has zero $\frac 12$-dimensional Hausdorff measure (note: another gauge function make the measure positive. I guess $r^{\beta}(\log \log r)^{1-\beta}$ just like zeros of stable processes). Is it true that zeros of SRW and also in Theorem~\ref{thm:image}, the Hausdorff measure is zero? Idea: If $\card{\Psi\cap [0,r]}/r^{\beta} \rightarrow 0$ a.s., then Lemma~\ref{lem:MDPstronger2} implies the claim. \textbf{Guess:} The exact gauge function is $\sqrt{2r \log \log r}$ just like zeros of Brownian motion\footnote{\invisible{[The exact Hausdorff measure of the level sets of Brownian motion]}}. \textbf{Update:} I think the paper that Jean Bertoin said can be used to obtain the exact dimension function.}
	
	
	\textbf{Later:} Zeros of two-dimensional SRW is also an interesting set. Guess: $dim=\frac 14$ (or maybe 0?). For larger spaces, the zeros are finite.}
}

\subsubsection{The Graph of the Simple Random Walk}
\label{subsec:srw-graph}

The graph of the random walk $(S_n)_{n\in\mathbb Z}$ is 
$
	\Psi:=\{(n,S_n):n\in\mathbb Z\}\subseteq\mathbb R^{k+1}.
$
{It can be seen that $\Psi$ is a point-stationary point process, and hence, $[\Psi,0]$ is unimodular (see Subsection~I.4.3.1).
}
Since $\card{\Psi\cap [-n,n]^{k+1}}\leq 2n+1$, 
the mass distribution principle (Theorem~\ref{thm:mdp-simple}) implies that $\dimH{\Psi}\leq 1$.
In addition, if $S_1$ has finite first moment, then 
the strong law of large numbers implies that $\lim_n \frac 1 n S_n=\omid{S_1}$.
This implies that $\liminf_n \frac 1 n \card{\left(\Psi\cap [-n,n]^{k+1}\right)} >0$.
Therefore, the unimodular Billingsley lemma (Theorem~\ref{thm:billingsley})
implies that $\dimH{\Psi}\geq 1$. Hence, $\dimH{\Psi}=1$.
\invisible{(Later: Question: What about the Minkowski dimension? What about the heavy tailed case?)} 


Below, the focus is on the case $k=1$ and on the following metric:
\begin{equation}
\label{eq:anothermetric}
d((x,y),(x',y')):=\max \{ \sqrt{\norm{x-x'}}, {\norm{y-y'}}\}.
\end{equation}

Theorem~I.3.31 implies that, by considering this metric,
unimodularity is preserved and dimension is not decreased. 
Under this metric, the ball $N_n(0)$ is $\Psi\cap [-n^2,n^2]\times [-n, n]$.
{It is straightforward that \del{Note that the whole $\mathbb Z^2$ has an equivariant disjoint $n$-covering similar to the one in
Subsection~\ref{subsec:euclidean}. By using {Theorem~\ref{thm:mdp-simple}},
one readily obtains that }$\mathbb Z^2$ has growth rate 3 and also} Minkowski and Hausdorff dimension 3 under this metric.
	
\begin{proposition}
If the jumps are $\pm 1$ uniformly, under the metric~\eqref{eq:anothermetric},
the graph $\Psi$ of the simple random walk satisfies $\dimM{\Psi}= \dimH{\Psi}=2.$
\end{proposition}
	
\begin{proof}
Let $n\in \mathbb N$. The ball $N_n(0)$ has at most $2n^2+1$ elements. So the mass distribution principle
(Theorem~\ref{thm:mdp-simple}) implies that $\dimH{\Psi}\leq 2$. For the other side, let $\mathcal C$ be
the equivariant disjoint covering of $\mathbb Z^2$ by translations of the rectangle $[-n^2,n^2]\times [-n,n]$
(similar to Example~I.3.12). For each rectangle $\sigma\in \mathcal C$,
select the right-most point in $\sigma\cap \Psi$ and let $\bs S=\bs S_{\Psi}$ be the set of selected points.
By construction, $\bs S$ gives an $n$-covering of $\Psi$ and it can be seen that it is an equivariant covering.
Let $\sigma_0$ be the rectangle containing the origin. By construction, $0\in \bs S$ if and only if it is
either on a right-edge of $\sigma_0$ or on a horizontal edge of $\sigma_0$ and the random walk stays outside $\sigma_0$. 
The first case happens with probability $1/(2n^2+1)$. By classical results concerning the hitting time of random walks,
one can obtain that the probability of the second case lies between two constant multiples of $n^{-2}$.
It follows that $\myprob{0\in \bs S}$ lies between two constant multiples of $n^{-2}$. 
Therefore, $\dimMl{\Psi}\geq 2$. This proves the claim.
\end{proof}

\del{\begin{remark}
One can generalize the above proposition by allowing $S_n\in\mathbb R^k$ and assuming that the jumps have
zero mean and finite second moments.
\end{remark}}

\subsection{A Drainage Network Model}

Let $[\bs T, \bs o]$ be the one-ended tree in Subsection~I.4.5 {equipped with the graph-distance metric}.
\todel{Recall that the set of vertices is the even lattice $\{(x,y)\in \mathbb Z^2: x+y \bmod 2 = 0\}$,
and the parent $F(x,y)$ of vertex $(x,y)$ is $(x\pm 1, y-1)$, where the sign is chosen uniformly and independently. 
The following theorem complements Theorem~I.4.14.}

\begin{proposition}
\label{thm:drainage}
\del{Under the graph-distance metric, }One has
$ \dimM{\bs T} = \dimH{\bs T} = \frac 3 2. $
\end{proposition}

\begin{proof}
Theorem~I.4.14 proves that $\dimM{\bs T}=\frac 3 2$. So it is enough to prove $\dimH{\bs T}\leq \frac 3 2$.
To use the unimodular Billingsley lemma, an upper bound on $\omid{\card{N_n(\bs o)}}$ is derived. Let 
$e_{k,l}:=\card{\left(F^{-k}(F^l(\bs o))\setminus F^{-(k-1)}(F^{l-1}(\bs o))\right)}$
be the number of descendants of order $k$ of $F^l(\bs o)$ which are not a descendant
of $F^{l-1}(\bs o)$ (for $l=0$, let it be just $\card{F^{-k}(\bs o)}$).
One has $\card{N_n(\bs o)} = \sum_{k,l} e_{k,l}\identity{\{k+l\leq n\}}$. 
\invisible{Later: Say why this holds?}
It can be seen that $\omid{e_{k,l}}$ is equal to the probability that two independent
paths of length $k$ and $l$ starting both at $\bs o$ do not collide at another point. 
Therefore, $\omid{e_{k,l}}\leq c (k\wedge l)^{-\frac 12}$ for some $c$ and all $k,l$.
This implies that (in the following, $c$ is updated at each step to a new constant without changing the notation)
\begin{eqnarray*}
\omid{\sum_{k,l\geq 0} e_{k,l} \identity{\{k+l\leq n\}}} &\leq & \sum_{k=0}^{\floor{\frac n 2}} ck^{-\frac 12}(n-k)
\le  cn\sum_{k=0}^{\floor{\frac n 2}} k^{-\frac 12}
\le cn^{\frac 3 2}.
\end{eqnarray*}
The above inequalities imply that $\omid{\card{N_n(\bs o)}}\leq cn^{\frac 3 2}$ for some $c$ and all $n$.
Therefore, the unimodular Billingsley lemma (Theorem~\ref{thm:billingsley}) implies that $\dimH{\bs T} \leq \frac 32$.
So the claim is proved.
\end{proof}

\subsection{{Self Similar Unimodular Spaces}}
\label{subsec:selfsimilar}

In this subsection, two examples are presented which have some kind of self-similarity heuristically,
but do not fit into the framework of Subsection~I.4.6.

\subsubsection{Unimodular Discrete Spaces Defined by Digit Restriction}
\label{subsec:digits}
Let $J\subseteq\mathbb Z^{\geq 0}$. For $n\geq 0$, consider the set of natural numbers with expansion
$(a_na_{n-1}\ldots a_0)$ in base 2 such that $a_i=0$ for every $i\not \in J$. \del{As in the definition
of the unimodular discrete Cantor set (Subsection~I.\ref{I-subsec:cantor}), }{Similarly to the examples in Subsection~I.4.6,} one can shift this set 
randomly and take a limit to obtain a unimodular discrete space.
This can be constructed in the following way as well: Let $\bs T_0:=\{0\}$.
If $n\in J$, let $\bs T_{n+1}:=\bs T_n\cup (\bs T_n\pm 2\times 2^n)$,
where the sign is chosen i.i.d., each sign with probability $1/2$.
If $n\not\in J$, let $\bs T_{n+1}:=\bs T_n$. Finally, let $\Psi:=\cup_n \bs T_n$.

The upper and lower asymptotic densities of $J$ in $\mathbb Z^{\geq 0}$ are
defined by $\densityU{}(J):=\limsup_n \frac 1 nJ_n$ and
$\densityL{}(J):=\liminf_n \frac 1nJ_n$, where $J_n:= \card{J\cap \{0,\ldots,n\}}$.

\begin{proposition}
\label{prop:digits}
Almost surely,
\begin{eqnarray*}
\dimH{\Psi} = \dimMu{\Psi} = \growthu{\card{N_n(\bs o)}} &=& \densityU{}(J),\\
\dimMl{\Psi} = \growthl{\card{N_n(\bs o)}} &=& \densityL{}(J).
\end{eqnarray*}
\end{proposition}
{In particular, this provides another example of a unimodular discrete space where the (polynomial) growth rate does not exist.}
\begin{proof}
Let $n\geq 0$ be given. Cover $\bs T_n$ by a ball of radius $2^n$ centered at the minimal element of $\bs T_n$.
By the same recursive definition, one can cover $\bs T_{n+1}$ by either 1 or 2 balls of the same radius. 
Continuing the recursion, an equivariant $2^n$-covering $\bs R_n$ is obtained.
It is straightforward to see that $\myprob{\bs R_n(\bs o)>0} = 2^{-J_n}$.
Since these coverings are uniformly bounded (Definition~I.3.9),
Lemma~I.3.10 implies that $\dimMl{\Psi} = \densityL{}(J)$ and $\dimMu{\Psi} = \densityU{}(J)$. 
One has 
\begin{equation}
\label{eq:prop:digits}
\card{\bs T_m} = 2^{J_m	}.
\end{equation}
This implies that $\card{N_{2^n}(\bs o)}\leq 2^{J_n+1}$. One can deduce that $\growthu{\card{N_n(\bs o)}} \leq \densityU{}(J)$.
So the unimodular Billingsley lemma (Theorem~\ref{thm:billingsley}) gives $\dimH{\Psi}\leq \densityU{}(J)$. This proves the claim.
\end{proof}

\subsubsection{Randomized Discrete Cantor set}
\label{subsec:cantor-randomized}
This  subsection proposes a unimodular discrete analogue of the \textit{random Cantor set}, recalled below. Let $0\leq p\leq 1$ and $b>1$.  The {random Cantor set} in $\mathbb R^k$~\cite{Ha81} (see also~\cite{bookBiPe17})
is defined by $\Lambda_k(b,p):=\cap_n E_n$, where $E_n$ is defined by the following
random algorithm: Let $E_0:=[0,1]^k$. For each $n\geq 0$ and each \textit{$b$-adic}
cube of edge length $b^{-n}$ in $E_n$, divide it into $b^k$ smaller $b$-adic cubes
of edge length $b^{-n-1}$. Keep each smaller $b$-adic cube with probability $p$
and delete it otherwise independently from the other cubes. Let $E_{n+1}$ be
the union of the kept cubes. It is shown in Section~3.7 of~\cite{bookBiPe17}
that $\Lambda_k(b,p)$ is empty for $p\leq b^{-k}$ and otherwise,
has dimension $k+\log_b p$ conditioned on being non-empty.

For each $n\geq 0$, let $\bs K_n$ be the set of {{lower left corners} of the $b$-adic cubes}
forming $E_n$. It is easy to show that $\bs K_n$ tends to $\Lambda_k(b,p)$ a.s. under the Hausdorff metric. 

\begin{proposition}
\label{prop:randomCantor}
		Let $\bs K'_n$ denote the random set obtained by biasing the distribution
		of $\bs K_n$ by $\card{\bs K_n}$ {(Definition~I.C.1)}.
		Let $\bs o'_n$ be {a point chosen uniformly at random in $\bs K'_n$.}
\begin{enumerate}[(i)]
\item \del{The distribution of }$[b^n \bs K'_n, \bs o'_n]$ converges {weakly} to some unimodular discrete space $[\hat{\bs K}, \hat{\bs o}]$. 
\item If $p<b^{-k}$, then $\hat{\bs K}$ is finite a.s., hence, $\dimH{\hat{\bs K}}=0$ a.s.
\item If $p\geq b^{-k}$, then $\hat{\bs K}$ is infinite a.s. and 
\[
\dimH{\hat{\bs K}} = \dimM{\hat{\bs K}} =k+\log_b p, \quad a.s.
\]
\end{enumerate}
\end{proposition}
	
Note that in contrast to the continuum analogue~\cite{Ha81}, for $p=b^{-k}$, the set is non-empty and even infinite,
though still zero dimensional. Also, for $p<b^{-k}$ the set is non-empty as well.
\invisible{\textbf{Later:} In the continuum case, if we bias by the volume of $E_n$ and then take limit, a nonempty set is obtained.}
	
To prove the above proposition, the following construction of $\hat{\bs K}$ will be used.
First, consider the usual nested sequence of partitions $\Pi_n$ of $\mathbb Z^k$ 
by translations of the cube $\{0,\ldots,b^n-1\}^k$, where $n\geq 0$. To make it \textit{stationary}, 
shift each $\Pi_n$ randomly as follows. Let $a_0,a_1,\ldots\in \{0,1,\ldots, b-1\}^k$
be i.i.d. uniform numbers and let $\bs U_n=\sum_{i=0}^n a_ib^i \in \mathbb Z^k$.
Shift the partition $\Pi_n$ by the vector $\bs U_n$ to form a partition denoted by $\Pi'_n$.
It is easy to see that $\Pi'_n$ is a nested sequence of partitions.
	
\begin{lemma}
\label{lem:randomCantor}
Let $(\Pi'_n)_n$ be the stationary nested sequence of partitions of $\mathbb Z^k$ defined above.
For each $n\geq 0$ and each cube $C\in \Pi'_n$ that does not contain the origin, with probability $1-p$
(independently for different choices of $C$), mark {all points in} $C\cap \mathbb Z^k$ for deletion. Then,
the set of the unmarked points of $\mathbb Z^k$, \rooted{} at the origin,
has the same distribution as $[\hat{\bs K}, \hat{\bs o}]$ defined in Proposition~\ref{prop:randomCantor}.
\end{lemma}

\begin{proof}[Proof of Lemma~\ref{lem:randomCantor}]
Let $\Phi$ be the set of unmarked points in the algorithm.
For $n\geq 0$, let $\bs C_n$ be the cube in $\Pi'_n$ that contains the origin.
It is proved below that $\bs C_n\cap \Phi$ has the same distribution as $b^n(\bs{K}'_n -\bs o_n)$. This \del{easily }implies the claim.
		
Let $A_n\subseteq [0,1]^k$ be the set of possible outcomes of $\bs o'_n$. One has $\card{A_n} = b^{kn}$.
For $v\in A_n$, it is easy to see that the distribution of $b^n(\bs{K}'_n -\bs o_n)$, conditioned
on $\bs o'_n=v$, coincides with the distribution of $\bs C_n\cap \Phi$ conditioned on $\bs C_n = b^n([0,1)^k-v)$.
So it remains to prove that $\myprob{\bs o'_n=v}= \myprob{\bs C_n = b^n([0,1)^k-v)}$, which is left to the reader.
\end{proof}

Here is another description of $\hat{\bs K}$. The nested structure of $\bigcup_n \Pi'_n$ defines a tree as follows.
The set of vertices is $\bigcup_n \Pi'_n$. For each $n\geq 0$, connect (the vertex corresponding to) every
cube in  $\Pi'_n$ to the unique cube in $\Pi'_{n+1}$ that contains it. This tree is the canopy tree
(Subsection~I.4.2.1) with offspring cardinality $N:=b^k$, except that the root
(the cube $\{0\}$) is always a leaf. Now, keep each vertex with probability $p$ and remove it
with probability $1-p$ in an i.i.d. manner. Let $\bs T$ be the connected component of the
remaining graph that contains the root. Conditioned on the event that $\bs T$ is infinite,
$\hat{\bs K}$ corresponds to the set of leaves in the connected component of the root.
	
\begin{proof}[Proof of Proposition~\ref{prop:randomCantor}]
The unimodular Billingsley lemma is used to get an upper bound on the Hausdorff dimension.
For this $\omid{\card{N_{b^n}(\bs o)}}$ is studied. Consider the tree $[\bs T, \bs o]$ defined above and
obtained by the percolation process on the canopy tree with offspring cardinality $N:=b^k$. Let $C$ be any cube
in $\Pi'_i$ that does not contain the origin. Note that the subtree of descendants of $C$ in the percolation cluster
(conditioned on keeping $C$) is a Galton-Watson tree with binomial offspring distribution with parameters $(N,p)$.
Classical results on branching processes say $\omidCond{\card{C\cap \hat{\bs K}}}{\Pi'_i} = p m^i$, where $m:=p b^k$.
So the construction implies that
\[
\omid{\card{\bs C_n\cap \hat{\bs K}}} = 1+ p(N-1)\left(m^{n-1}+m^{n-2}+\cdots + 1 \right).
\]
For $m>1$, the latter is bounded by $l m^n$
for some constant $l$ not depending on $n$. Note that $N_{b^n}(\bs o)$ is contained 
in the union of $\bs C_n$ and $3^k-1$ other cubes in $\Pi'_n$. It follows that 
$\omid{\card{N_{b^n}(\bs o)}}\leq l' m^n$, where $l'=l+(3^k-1)p$. So the unimodular Billingsley lemma
(Theorem~\ref{thm:billingsley}) implies that $\dimH{\hat{\bs K}}\leq k+\log_b p$. The claim for $m=1$ and $m<1$ are similar.
		
Consider now the Minkowski dimension. \del{By the above part of the proof, it is enough to assume $m>1$. }{As above, we assume $m>1$ and the proofs for the other cases are similar.}
Let $n\geq 0$ be given. By considering the partition $\Pi'_n$ by cubes, one can construct a $b^n$-covering
$\bs R_n$ as in Theorem~\ref{thm:lowerBoundR^d}. This covering satisfies
$
\myprob{\bs R_n(\bs o)\geq 0} = \omid{ 1/{\card{(\bs C_n \cap \hat{\bs K})}}}.
$
Let $[\bs T', \bs o']$ be the eternal Galton-Watson tree of Subsection~I.4.2.3 
with binomial offspring distribution with parameters $(N,p)$. By regarding $\bs T'$ as a family tree,
it is straightforward that $[\bs T, \bs o]$ has the same distribution as the part of $[\bs T', \bs o']$,
up to the generation of the root (see~\cite{eft} for more details on eternal family trees).
Therefore, Lemma~5.7 of~\cite{eft} implies that 
$
\omid{1/{\card{(\bs C_n \cap \hat{\bs K})}}} = m^{-n} \myprob{h(\bs o')\geq n}.
$
Since $m>1$, $\myprob{h(\bs o')\geq n}$ tends to the non-extinction probability of the descendants of the root, which is positive.
By noticing the fact that the radii of the balls are $b^n$ and the covering is uniformly bounded,
one gets that $\dimM{\hat{\bs K}}= \log_b m = k+\log_b p$.
		
Finally, it remains to prove that $\hat{\bs K}$ is infinite a.s. when $p=b^{-k}$.
In this case, consider the eternal Galton-Watson tree $[\bs T',\bs o']$ as above.
Proposition~6.8 of~\cite{eft} implies that the generation of the root is infinite a.s. This proves the claim.
\end{proof}

	\invisible{\begin{guess}
			If $b=2$ and $p$ is large enough, then the subgraph of $\mathbb Z^k$ induced by $\hat{\bs K}$ contains infinite connected components a.s. (see line -5 of page 114 of~\cite{bookBiPe17}).
		\end{guess}
		\textbf{Guess:} Intersection of $\hat{\bs K}$ with an independent point-stationary point process has a known dimension. See Theorem~5.1 and Corollary 5.3 of Section~7 of~\cite{bookBiPe17} (in the old book).
		
		NOTE: I thinks Randomized self similar sets, which are very more general, are similarly treated (analogous to [2002, A dimensional result for random self-similar sets]). I think this is good to mention somewhere. What do you think?}

\subsection{Cayley Graphs}
\label{subsec:cayley}

As mentioned in Subsection~I.3.6,  the dimension of a Cayley graph depends only on the group and not on the generating set. The following result connects it to the growth rate of the group. Note that 
Gromov's theorem~\cite{Gr81} implies that the polynomial growth degree exists and is either an integer or infinity.

\begin{theorem}
	\label{thm:Cayley}
For every finitely generated group $H$ with polynomial growth {degree} $\alpha\in [0,\infty]$, one has 
$
\dimM{H} = \dimH{H} = \alpha
$
{and $\measH{\alpha}(H)<\infty$}.
\end{theorem}
\begin{proof}
First, assume $\alpha<\infty$. The result of Bass~\cite{Ba72} implies that there are constants $c,C>0$
such that $\forall r\geq 1: cr^{\alpha}<\card{N_r(o)} \leq Cr^{\alpha}$, where $o$ is an arbitrary element of $H$.
So the mass distribution principle (Theorem~\ref{thm:mdp-simple}) and part~\eqref{part:lem:lowerbound:all}
of Lemma~\ref{lem:lowerboundiid} imply that $\dimM{H} = \dimH{H}=\alpha$. 
{In addition, \eqref{eq:thm:mdp-simple} in the proof of Theorem~\ref{thm:mdp-simple} implies that $\contentH{\alpha}{M}(H)\geq 1/C$ for all $M\geq 1$, which implies that $\measH{\alpha}(H)\leq C<\infty$.}
	
Second, assume $\alpha=\infty$. The result of \cite{VaWi84} shows that for any $\beta<\infty$,
$\card{N_r(o)}>r^{\beta}$ for sufficiently large $r$\invisible{ (see~(1.10) in~\cite{VaWi84})}.
Therefore, part~\eqref{part:lem:lowerbound:all} of Lemma~\ref{lem:lowerboundiid} implies
that $\dimMl{H}\geq \beta$. Hence, $\dimM{H}=\dimH{H}=\infty$ and the claim is proved.
\end{proof}

{
It is natural to expect that $\measH{\alpha}(H)>0$ as well, but only a weaker inequality will be proved in Proposition~\ref{prop:Cayley}. 
}


\invisible{\mar{Which one of these to include?\\ Mention them briefly?}\textbf{Question:} Can we calculate the Hausdorff measure?
	
	\textbf{Question:} What about deterministic transitive unimodular graphs? \textbf{Update:} Tom Hutchcroft said that there is a result similar to Gromov's result for transitive unimodular graphs. So we might consider devoting a few paragraphs on transitive graphs.
}

\invisible{

{
%
%
%
%
%
%
%
	
	{	
	\textbf{Idea:} More interesting than polynomial growth, are groups of intermediate growth. The growth is usually of the form $e^{n^{\alpha}}$. We can try to use another \textit{gauge function}: In the definition of Hausdorff dimension, to each ball of radius $r$, assign $e^{r^{\alpha}}$. Try the lamplighter group. Or more interestingly, Grigorchuk's group or groups of \textit{oscillating growth} [???].
	
	\textbf{Question:} Is there a natural metric on the lamplighter group that makes it finite dimensional (note that Folner sets are not balls here).
	}
}}

\subsection{Notes and Bibliographical Comments}

{The proof of {Proposition~\ref{prop:tree-expGrowth}} was suggested by R. Lyons.} 
Bibliographical comments on {some of} the examples discussed in this
section can be found at the end of Section I.4. {The example defined by digit restriction (Subsection~\ref{subsec:digits}) is inspired by an example in the continuum setting (see e.g., Examples~1.3.2 of~\cite{bookBiPe17}). The randomized discrete Cantor set (Subsection~\ref{subsec:cantor-randomized}) is inspired by the \textit{random cantor set} (see e.g., Section~3.7 of~\cite{bookBiPe17}).}

\invisible{
{\mar{I suggest to delete this since it is a duplicate of the proof of Lemma~III.\ref{III-lem:nonamenable-graph}. Instead, we can say that Theorem~\ref{thm:tree-expGrowth} is implied by Part III}} \sout{
Here is another proof of Lemma \ref{lem:tree-expGrowth},
which is similar to that of Lemma~III.\ref{III-lem:nonamenable-graph}.
Part~(iii) of Theorem~8.13 of~\cite{processes} implies that there exists an equivariant subtree $\bs T'$ of $\bs T$ 
with positive \textit{isoperimetric constant} a.s.
In particular, this gives that $\card{N_{n+1}(\bs T',\bs o)}/\card{N_n(\bs T',\bs o)}$
is bounded from below by a positive constant (that might depend on the realization of $\bs T'$).
This proves that $\bs T'$ has exponential growth. So $\bs T$ has exponential growth as well.
}
}

\section{Frostman's Theory}
\label{sec:frostman}

This section provides a unimodular version of Frostman's lemma and some of its applications.
In a sense to be made precise later, this lemma gives converses to the mass distribution principle\del{ and
the unimodular Billingsley lemma (see in particular, Corollary~\ref{cor:Hwei} below)}.
It is a powerful tool \del{to prove theoretical results regarding }{in the theoretical analysis of} the unimodular Hausdorff dimension.
For example, it is used in this section to derive inequalities for the dimension of product spaces 
and \textit{embedded spaces} (Subsections~\ref{subsec:product} and~\ref{subsec:embedded}).
It is also the basis of many of the results in~\cite{I}.


\subsection{Unimodular Frostman Lemma}
\label{subsec:frostman-general}

The statement of the unimodular Frostman lemma requires the definition
of \textit{weighted Hausdorff} content. {The latter is based on the notion of \textit{equivariant weighted collections of balls} as follows.} For this, the following mark space is needed.
Let $\Xi$ be the set of functions $c:\mathbb R^{\geq 0}\to \mathbb R^{\geq 0}$ which are positive in {only} finitely many points;
i.e., $c^{-1}((0,\infty))$ is a finite set.  
{Remark~\ref{rem:frostman-markspace} below defines a metric on $\Xi$ such that the notion of $\Xi$-valued equivariant processes (Definition~I.2.8) are well defined.}
%
%
%
{Such a process $\bs c$ is called}
an \textbf{equivariant weighted collection of balls}
\footnote{The term `weighted' refers to the weighted sums in Definition~\ref{def:xi^alpha} and should not be confused
with equivariant weight functions of Definition~\ref{def:weight}}.
Consider a unimodular discrete space $[\bs D, \bs o]$ with distribution $\mu$.
For $v\in \bs D$, the reader can think of the value $\bs c_r(v):=\bs c(v)(r)$, if positive, to indicate that
there is a ball in the collection, with radius $r$, centered at $v$, and with {\textbf{cost} (or weight)} $\bs c_r(v)$.
Note that extra randomness is allowed in the definition.
A ball-covering $\bs R$ can be regarded a special case of this construction by letting
$\bs c_r(v)$ be 1 when $r=\bs R(v)$ and 0 otherwise.

\begin{definition}
\label{def:xi^alpha}
Let $f:\mathcal D_*\rightarrow\mathbb R$ be a measurable function and $M\geq 1$. 
An equivariant weighted collection of balls $\bs c$ is called a \defstyle{$(f,M)$-covering} if 
\begin{equation}
\label{eq:weightedCovering}
\forall v\in \bs D: f(v) \leq \sum_{u\in\bs D}\sum_{r\geq M} \bs c_r(u)\identity{\{v\in N_r(u)\}}, \quad a.s.,
\end{equation}
where $f(v):=f[\bs D,v]$ for $v\in\bs D$. For $\alpha\geq 0$, define 
\begin{eqnarray*}
	\xi^{\alpha}_M(f)&:=&\inf\left\{ \omid{\sum_{r} \bs c_r(\bs o)r^{\alpha}}: \bs c \text{ is a } (f,M)\text{-covering}\right\},\\
	\xi^{\alpha}_{\infty}(f)&:=& \lim_{M\rightarrow\infty} \xi^{\alpha}_M(f).
\end{eqnarray*}
%
\end{definition}

It is straightforward that every equivariant ball-covering of Definition~I.3.15
gives a $(1,1)$-covering, where $1$ is regarded as the constant function $f\equiv 1$ on $\dstar$. This gives {(see also Conjecture~\ref{conj:frostman} below)}
\begin{equation}
\label{eq:frostman-xi<H}
\xi^{\alpha}_M(1) \leq \contentH{\alpha}{M}(\bs D).
\end{equation}

Also, by considering the case $\bs c_M(v) := f(v)\vee 0$, one can see that if $f\in L_1(\mathcal D_*,\mu)$, then
\[
\xi^{\alpha}_M(f)\leq M^{\alpha} \omid{f(\bs o)\vee 0} <\infty.
\]
In the following theorem, to be consistent with the setting of the paper,
the following notation is used: $w(u):=w([D,u])$ for $u\in D$, and $w(N_r(v))=\sum_{u \in N_r(v)} w(u)$.
Also, recall that a deterministic equivariant weight function is given by a measurable function
$w:\dstar\to\mathbb R^{\geq 0}$ (see Example~I.2.10).  

\begin{theorem}[Unimodular Frostman Lemma]
\label{thm:frostmanGeneral}
Let $[\bs D, \bs o]$ be a unimodular discrete space, $\alpha\geq 0$ and $M\geq 1$.
\begin{enumerate}[(i)]
\item \label{thm:frostmanGeneral:1}
There exists a bounded measurable weight function $w:\mathcal D_*\rightarrow\mathbb R^{\geq 0}$ such that {$\omid{w(\bs o)} = \xi^{\alpha}_M(1)$ and} almost surely,
\begin{equation}
\label{eq:frostman:1}
\forall v\in \bs D,\  \forall r\geq M:  w(N_r(v))\leq r^{\alpha}.
\end{equation}
\item \label{thm:frostmanGeneral:2}
Given a non-negative function $h\in L_1(\mathcal D_*,\mu)$, the {first} condition can be replaced by 
$
\omid{w(\bs o)h(\bs o)} = \xi^{\alpha}_M(h). 
$
\item \label{thm:frostmanGeneral:3}
In the setting of~\eqref{thm:frostmanGeneral:2}, if $\bs D$ has finite $\alpha$-dimensional Hausdorff measure {(e.g., when $\dimH{\bs D}>\alpha$)}
and $h\not\equiv 0$, then $w[\bs D,\bs o]\neq 0$ with positive probability.
\end{enumerate}
\end{theorem}

The proof is given later in this subsection.

\todel{\begin{corollary}
\label{cor:Hwei}
For all unimodular discrete spaces $[\bs D, \bs o]$ and all $\epsilon>0$,
there exists a deterministic equivariant weight function $w$ such that
\[
\growthu{w(N_r(\bs o))}-\epsilon \leq {\dimH{\bs D}} \leq \growthu{w(N_r(\bs o))}.
\]
In addition, if $\bs D$ has finite $\dimH{\bs D}$-dimensional Hausdorff measure, then $w$ can be chosen such that
\[
\dimH{\bs D} = \growthu{w(N_r(\bs o))}.
\]
\end{corollary}
\begin{proof}
If $\dimH{\bs D}=\infty$, then let $w(\cdot)\equiv 1$. In this case, the claim follows from 
the unimodular Billingsley lemma (Theorem~\ref{thm:billingsley}). So assume $\dimH{\bs D}<\infty$ and let
$\alpha:=\dimH{\bs D}+\epsilon$. One has $\measH{\alpha}(\bs D)=0$ (Lemma~I.\ref{I-lem:Hmeas}).
So, by part~\eqref{thm:frostmanGeneral:3} of the unimodular Frostman lemma,
the function $w$ in the lemma is not identical to zero. Therefore, the unimodular Billingsley lemma
implies that $\dimH{\bs D}\leq \growthu{w(N_r(\bs o))}$. On the other hand,
\eqref{eq:frostman:1} implies that $\growthu{w(N_r(\bs o))}\leq \alpha = \dimH{\bs D}+\epsilon$. This proves the claim.
\end{proof}
}

\begin{remark}
\label{rem:frostman-ineq}
One can show that \eqref{eq:frostman:1} implies that $\omid{\bs w(\bs o)}  \leq  \xi^{\alpha}_M(1)$ and $\omid{\bs w(\bs o)h(\bs o)} \leq  \xi^{\alpha}_M(h)$.
Therefore, the (deterministic) weight function $w$ given in the unimodular Frostman lemma
is a maximal equivariant weight function satisfying~\eqref{eq:frostman:1}
(it should be noted that \del{the converse of this claim is not true }{the maximal is not unique}). 
The proof is similar to that of the mass distribution principle (Theorem~\ref{thm:mdp-simple}) and is left to the reader.
\end{remark}

\begin{conjecture}
\label{conj:frostman}
One has $\contentH{\alpha}{M}(\bs D)=\xi^{\alpha}_M(1)$.
\invisible{{\textbf{Note}: This holds for compact continuum spaces. For point-stationary point processes in $\mathbb R^k$, we have shown that $\contentH{\alpha}{1}(\bs D)\leq C_k \xi^{\alpha}_1$, where $C_k$ only depends on $k$.}}
\end{conjecture}

{Lemma~\ref{lem:frostmanAuxiliary} below proves a weaker inequality.}
{For instance, it can be seen that the conjecture holds for $\mathbb Z^k$ (with the $l_{\infty}$ metric) and for the non-ergodic example of  Example~I.3.20. In the former, this is obtained by considering the constant weight function $w(\cdot)\equiv \left(\frac{M}{2M+1}\right)^k$, which satisfies the claim of the unimodular Frostman lemma. The latter is similar by letting $w[\mathbb Z,0]:=\frac M{2M+1}$ and $w[\mathbb Z^2,0]:=0$.}

\del{
\begin{example}
Assume $[\bs D, \bs o]:=[\mathbb Z^k, 0]$ is equipped with the $l_{\infty}$ metric and let $M\in\mathbb N$.
By the proof of Proposition~I.3.29, one can see that 
\[\xi^k_M(1)\leq \contentH{k}{M}(\mathbb Z^k)\leq \left(\frac{M}{2M+1}\right)^k.\]
Let $w(\cdot)\equiv \left(\frac{M}{2M+1}\right)^k$. This weight function satisfies~\eqref{eq:frostman:1}
for $\alpha=k$ and also $\omid{w(\bs o)}\geq \xi^k_M(1)$. Therefore, by Remark~\ref{rem:frostman-ineq} above,
$\omid{w(\bs o)}=\xi^k_M(1)$. So $w$ satisfies the claim of the unimodular Frostman lemma.
Note that Conjecture~\ref{conj:frostman} holds in this case.
\end{example}

\begin{example}
\label{ex:frostman-nonergodic}
Assume $\bs D$ is $\mathbb Z$ with probability $\frac 12$ and $\mathbb Z^2$ with probability $\frac 12$
(see Example~I.3.20). Given $M\in\mathbb N$, let $w[\mathbb Z,0]:=\frac M{2M+1}$ and
$w[\mathbb Z^2,0]:=0$. As in the previous example, one can show that $w$ satisfies the claim of the
unimodular Frostman lemma and Conjecture~\ref{conj:frostman} holds in this case.
\end{example}
}

\begin{remark}
	\label{rem:frostman-mdp-billingsley}
	{The unimodular Frostman lemma implies that, in theory, the mass transport principle (Theorem~\ref{thm:mdp-simple}) is enough for bounding the Hausdorff dimension from above. However, there are very few examples in which the function $w$ given by the unimodular Frostman lemma can be explicitly computed (in some of the examples, a function $w$ satisfying {only}~\eqref{eq:frostman:1} can be found; e.g., for two-ended trees). Therefore, in practice, the unimodular Billingsley lemma is more useful than the mass transport principle.}
	\invisible{1. In fact, for very special two-ended trees, one can construct the $w$ in Frostman.\\ 2. Another example, replace each point with a random number of points.} 
\end{remark}

\invisible{\textbf{Update:} (is this worth to write?) If we have a weight function $w$ such that $\growthu{w(N_n(\bs o))}<\alpha$ a.s., then the following works for all balls (but is not exactly the outcome of the unimodular Frostman lemma): $w'(v):=w(v) / (\sup_n w(N_n(v))/n^{\alpha})$. Does this works for zeros of SRW? }

The following lemma is needed to prove Theorem~\ref{thm:frostmanGeneral}.

\begin{lemma}
The function $\xi^{\alpha}_M: L_1(\mathcal D_*,\mu)\rightarrow\mathbb R$ is continuous.
In fact, it is $M^{\alpha}$-Lipschitz; i.e., 
$
\norm{\xi^{\alpha}_M(f_1) - \xi^{\alpha}_M(f_2)} \leq M^{\alpha} \omid{\norm{f_1(\bs o)-f_2(\bs o)}}.
$
\end{lemma}
\begin{proof}
Let $\bs c$ be an equivariant weighted collection of balls satisfying~\eqref{eq:weightedCovering} for $f_1$.
Intuitively, add a ball of radius $M$ at each point $v$ with cost $\norm{f_2(v)-f_1(v)}$.
More precisely, let $\bs c'_r(v):=\bs c_r(v)$ for $r\neq M$ and $\bs c'_M(v):=\bs c_M(v)+\norm{f_2(v)-f_1(v)}$.
This definition implies that $\bs c'$ satisfies~\eqref{eq:weightedCovering} for $f_2$. Also,
\[
\xi^{\alpha}_M(f_2)\leq \omid{\sum_{i} \bs c'_i(\bs o)i^{\alpha}} = \omid{\sum_{r} \bs c_r(\bs o)i^{\alpha}} + M^{\alpha}\omid{\norm{f_2(\bs o)-f_1(\bs o)}}.
\]
Since $\bs c$ was arbitrary, one obtains 
$
\xi^{\alpha}_M(f_2)\leq \xi^{\alpha}_M(f_1) + M^{\alpha}\omid{\norm{f_2(\bs o)-f_1(\bs o)}},
$
which implies the claim.
\end{proof}

\del{The following proof uses the ideas of Thm~8.17 of~\cite{bookMa95}. \mar{{This is already said in the bib notes.}}}
\begin{proof}[Proof of Theorem~\ref{thm:frostmanGeneral}]
	Part~\eqref{thm:frostmanGeneral:1} is implied by Part~\eqref{thm:frostmanGeneral:2} which is proved below.
	It is easy to see that $\xi^{\alpha}_M(tf)=t\xi^{\alpha}_M(f)$ for all $f$ and $t\geq 0$ and also 
	$
	\xi^{\alpha}_M(f_1+f_2) \leq \xi^{\alpha}_M(f_1)+\xi^{\alpha}_M(f_2)
	$
	for all $f_1,f_2$. Let $h\in L_1(\mathcal D_*,\mu)$ be given. By the Hahn-Banach theorem
	(see Theorem~3.2 of~\cite{bookRu73}), there is a linear functional $l:L_1(\mathcal D_*,\mu)\rightarrow\mathbb R$ such that 
	$
	l(h)=\xi^{\alpha}_M(h)
	$
	and
	$
	-\xi^{\alpha}_M(-f) \leq l(f) \leq \xi^{\alpha}_M(f), \del{\quad \forall f\in L_1.}
	$
	{for all $f\in L_1$.}
	Since $l$ is sandwiched between two functions which are continuous  at $0$ and are equal at $0$,
	one gets that $l$ is continuous at 0. Since $l$ is linear, 
	this implies that $l$ is continuous. Since the dual of $L_1(\mathcal D_*, \mu)$ is $L_{\infty}(\mathcal D_*, \mu)$,
	one obtains that there is a function $w\in L_{\infty}(\mathcal D_*, \mu)$ such that  
	$
	l(f) = \omid{f(\bs o)w(\bs o)}, \del{\quad \forall f\in L_1.}
	$
	{for all $f\in L_1$.}
	Note that if $f\geq 0$, then $\xi^{\alpha}_M(-f) = 0$ and so $l(f)\geq 0$.
	This implies that $w(\bs o)\geq 0$ a.s. (otherwise, let $f(\bs o):=\identity{\{w(\bs o)<0\}}$ to get a contradiction). 
        Consider a version of $w$ which is nonnegative everywhere.
	The claim is that $w$ satisfies the requirements\del{ of~\eqref{thm:frostmanGeneral:2}}. 
	
	Let $r\geq M$ be fixed.
	For all discrete spaces $D$, let $\bs S:=\bs S_D:=\{v\in D: w(N_r(v))>r^{\alpha}\}$.
	Define $f_r(v):= \card{N_r(v)\cap \bs S}$. By the definition of $\bs S_D$, one has
	\begin{equation}
	\label{eq:thm:frostmanGeneral:1}
	\omid{w(N_r(\bs o))\identity{\{\bs o\in \bs S\}}} \geq r^{\alpha}\myprob{\bs o \in \bs S}.
	\end{equation}	
	Moreover, if $\myprob{\bs o \in \bs S}>0$, then the inequality is strict.
	On the other hand, by the mass transport principle for the function
	$(v,u)\mapsto w(u)\identity{\{v\in \bs S\}} \identity{\{u\in N_r(v) \}}$, one gets
	\begin{eqnarray*}
		\omid{w(N_r(\bs o))\identity{\{\bs o\in \bs S\}}} &=& \omid{w(\bs o) \card{N_r(\bs o)\cap \bs S}}\\
		&=& \omid{w(\bs o)f_r(\bs o)}\\
		&=& l(f_r)\\
		&\leq & \xi^{\alpha}_M(f_r)\\
		&\leq & r^{\alpha} \myprob{\bs o \in \bs S},
	\end{eqnarray*}
	where the last inequality is implied by considering the following weighted collection of balls 
	for $f_r$: put balls of radius $r$ with cost 1 centered at the points in $\bs S$. 
	More precisely, let $\bs c_r(v):=\identity{\{v\in \bs S\}}$ and $\bs c_s(v):=0$ for $s\neq r$.
	It is easy to see that this satisfies~\eqref{eq:weightedCovering} for $f_r$,
	which implies the last inequality by the definition of $\xi^{\alpha}_M(\cdot)$.
	Thus, equality holds in~\eqref{eq:thm:frostmanGeneral:1}. Hence, $\myprob{\bs o\in \bs S}=0$;
        i.e., $w(N_r(\bs o))\leq r^{\alpha}$ a.s. Lemma~I.2.14 implies that almost surely,
        $\forall v\in \bs D: w(N_r(v))\leq r^{\alpha}$. So the same holds for all rational $r\geq M$ simultaneously.
	By monotonicity {of $w(N_r(v))$ w.r.t. $r$}, one gets that the latter almost surely holds for all $r\geq M$ as desired.
	
	Also, one has 
	$
	\omid{w(\bs o)h(\bs o)} = l(h) = \xi^{\alpha}_M(h).
	$
	Thus, $w$ satisfies the desired requirements.
	
	To prove~\eqref{thm:frostmanGeneral:3}, assume $h\not\equiv 0$ and $\measH{\alpha}(\bs D)<\infty$.
        By Lemma~I.3.25, one has $\contentH{\alpha}{M}(\bs D)>0$.
	So Lemma~\ref{lem:frostmanAuxiliary} below implies that $\xi^{\alpha}_M(h)>0$.
        Now, the above equation implies that $w$ is not identical to zero.
\end{proof}

The above proof uses the following lemma.

\begin{lemma}
\label{lem:frostmanAuxiliary}
	
Let $[\bs D, \bs o]$ be a unimodular discrete metric space.
\begin{enumerate}[(i)]
\item \label{lem:frostmanAuxiliary:1}
By letting $b:=\xi^{\alpha}_1(1)$, one has 
$
b\leq \contentH{\alpha}{1}(\bs D)\leq b + {b}{\norm{\log b}}.
$
\item \label{lem:frostmanAuxiliary:2}
Let $0\neq h\in L_1(\mathcal D_*,\mu)$ be a non-negative function. 
For $M\geq 1$, one has
\[
\contentH{\alpha}{M}(\bs D)\leq \inf_{a\geq 0}\left\{M^{\alpha}\omid{e^{-a h(\bs o)}}  + a \xi^{\alpha}_M(h)\right\}.
\]	
\item \label{lem:frostmanAuxiliary:3}
In addition, $\xi^{\alpha}_M(h)=0$ if and only if $\contentH{\alpha}{M}(\bs D)=0$.	
\end{enumerate}
\end{lemma}
\begin{proof}
\eqref{lem:frostmanAuxiliary:1}. 
By considering the cases where $\bs c(\cdot)\in \{0,1\}$, the first inequality is easily
obtained from the definition of $\xi^{\alpha}_1(1)$. In particular, this implies that $b\leq 1$.
The second inequality is implied by part~\eqref{lem:frostmanAuxiliary:2} by letting $h(\cdot):=1$ and $a:=-\log b\geq 0$.

\eqref{lem:frostmanAuxiliary:2}. 
Let $b'>\xi^{\alpha}_M(h)$ be arbitrary. So there exists an equivariant weighted collection
of balls $\bs c$ 
that satisfies~\eqref{eq:weightedCovering} for $h$ and 
$
\omid{\sum_{r\geq M} \bs c_r(\bs o)r^{\alpha}} \leq b'.
$
Next, given $a\geq 0$, define an equivariant covering $\bs R$ as follows.
For each $v\in \bs D$ and $r\geq M$ such that $\bs c_r(v)>0$, put a ball of radius $r$ at $v$ with probability
$a \bs c_r(v) \wedge 1$. Do this independently for all $v$ and $r$
(one should condition on $\bs D$ first). If more than one ball is put at $v$,
keep only the one with maximum radius. Let $\bs S$ be the union of the chosen balls.
For $u\in\bs D\setminus \bs S$, put a ball of radius $M$ at $u$. This gives an equivariant covering,
namely $\bs R$, by balls of radii at least $M$. Then, one can easily get
\begin{equation}
\label{eq:lem:frostmanGeneral:2}
\omid{\bs R(\bs o)}^{\alpha} \leq M^{\alpha}\myprob{\bs o\not\in \bs S} + \omid{\sum_{r\geq } (a\bs c_r(\bs o)\wedge 1)r^{\alpha}} \leq M^{\alpha} \myprob{\bs o\not\in \bs S} + ab'.
\end{equation}
	
To bound $\myprob{\bs o\not\in \bs S}$, consider a realization of $[\bs D, \bs o]$.
First, if for some $v\in \bs D$ and $r\geq M$, one has $a \bs c_r(v)>1$ and $\bs o\in N_r(v)$,
then $\bs o$ is definitely in $\bs S$. Second, assume this is not the case.
By~\eqref{eq:weightedCovering}, one has $\sum_{u\in\bs D}\sum_{r\geq M} \bs c_r(u)\identity{\{\bs o\in N_r(u)\}}\geq h(\bs o)$.
This implies that the probability that $\bs o\not\in \bs S$ in this realization is
\begin{eqnarray*}
\prod_{(v,r): \bs o\in N_r(v)} (1-a \bs c_r(v))
\leq  \mathrm{exp}\left( -\sum_{(v,r): \bs o\in N_r(v)} a\bs c_r(v) \right)
\leq e^{-a h(\bs o)}.
\end{eqnarray*}
In both cases, one gets $\myprob{\bs o\not\in \bs S} \leq \omid{e^{-a h(\bs o)}}$. 
Thus, \eqref{eq:lem:frostmanGeneral:2} implies 
that $\omid{R(\bs o)}^{\alpha} \leq M^{\alpha} \omid{e^{-a h(\bs o)}} + ab'$.
Since $a\geq 0$ and $b'>b$ are arbitrary, the claim follows.
	
\eqref{lem:frostmanAuxiliary:3}. 
Assume $\xi^{\alpha}_M(h)=0$. By letting $a\rightarrow\infty$ and using the first claim,
one obtains that $\contentH{\alpha}{M}(\bs D)=0$. Conversely, assume $\contentH{\alpha}{M}(\bs D)=0$. 
The first inequality in part~\eqref{lem:frostmanAuxiliary:1} gives that $\xi^{\alpha}_M(a)=0$ for any constant $a$.
Therefore, $\xi^{\alpha}_M(h)\leq \xi^{\alpha}_M(a)+\xi^{\alpha}_M((h-a)\vee 0))\leq M^{\alpha}\omid{(h-a)\vee 0}$.
By letting $a$ tend to infinity, one gets $\xi^{\alpha}_M(h)=0$.
\end{proof}

\begin{remark}
	\label{rem:frostman-markspace}
	In this subsection, the following metric is used on the mark space $\Xi$. Let $\Xi'$ be the set of {finite measures on} $\mathbb R^2$.
	By identifying $c\in \Xi$ with the {counting measure on the} finite set $\{(x,c(x)): x\in\mathbb R^{\geq 0}, c(x)>0\}$,
	one can identify $\Xi$ with a {Borel} subset of $\Xi'$. It is well known that $\Xi'$ is a complete separable metric space under {the {Prokhorov} metric} (see e.g., \cite{bookDaVe03I}).
	So one can define the notion of $\Xi'$-valued equivariant processes as in Definition~I.2.8.
	Therefore, $\Xi$-valued equivariant processes also make sense.
\end{remark}

\invisible{\textbf{Remark:} If sample dimension is not constant, then $w$ is zero on samples with higher dimension than $\dimH{\bs D}$.
	
	\textbf{Corollary:} If $\alpha:=\dimH{\bs D}$ is not an atom for the distribution of sample dimension, then $\measH{\alpha}(\bs D)=\infty$ (see Lemma~III.\ref{III-lem:essentialInf}).}

\subsection{Max-Flow Min-Cut Theorem for Unimodular One-Ended Trees}
\label{subsec:maxflow}

The result of this subsection is used in the next subsection for a Euclidean version of the unimodular Frostman lemma, but is of independent interest as well.
\del{This subsection provides a {result for unimodular one-ended trees which is analogous to} the max-flow min-cut theorem.
This result is used in the next subsection for a Euclidean version of the unimodular Frostman lemma, but is of independent interest.}

The max-flow min-cut theorem is a celebrated result in the field of graph theory (see e.g., \cite{FoFu62}).
In its simple version, it studies the minimum number of edges in a \textit{cut-set} in a finite graph;
i.e., a set of edges the deletion of which disconnects two given subsets of the graph.
A generalization of the theorem in the case of trees is obtained by considering cut-sets separating a given 
finite subset from the set of ends of the tree. This generalization is used to prove a version of
Frostman's lemma for compact sets in the Euclidean space (see e.g., \cite{bookBiPe17}).


This subsection presents an analogous result for unimodular one-ended trees.
It discusses cut-sets separating the set of leaves from the end of the tree. 
Since the tree has infinitely many leaves a.s. (see e.g., \cite{eft}), infinitely many edges
are needed in any such cut-set. Therefore, cardinality cannot be used to study minimum cut-sets.
The idea is to use unimodularity for a quantification of the size of a cut-set. 

Let {$[\bs T, \bs o; \bs c]$} be a unimodular {marked} one-ended tree with mark space $\mathbb R^{\geq 0}$.  
Assume the mark $\bs c(e)$ of each edge $e$ is well defined and call it
the \textbf{conductance} of $e$. Let $\bs L$ be the set of leaves of $\bs T$. As in Subsection~I.4.1.2,
let $F(v)$ be the parent of vertex $v$ and $D(v)$ be the descendants subtree of $v$.
\del{The new vocabulary used in the following definitions is that of the max-flow min-cut theorem.}

\begin{definition}
A \textbf{legal equivariant flow} on $[\bs T;\bs c]$ is an equivariant way of 
assigning extra marks $\bs f(\cdot)\in\mathbb R$ to the edges (see Definition~I.2.8
and Remark~I.2.12), such that almost surely,
\begin{enumerate}[(i)]
	\item for every edge $e$, one has $0\leq \bs f(e)\leq \bs c(e)$,
	\item for every vertex $v\in \bs T\setminus\bs L$, one has 
	\begin{equation}
	\label{eq:flowcons}
	\bs f(v,F(v))=\sum_{w\in F^{-1}(v)}\bs f(w,v).
	\end{equation}
	\end{enumerate}
	Also, an \textbf{equivariant cut-set} is an equivariant subset $\Pi$ of the edges of $[\bs T; \bs c]$
	that separates the set of leaves $\bs L$ from the end in $\bs T$. 
\end{definition}

Note that extra randomness is allowed in the above definition. 
The reader can think of the value $\bs f(v,F(v))$ as the \textit{flow} from $v$ to $F(v)$.
So~\eqref{eq:flowcons} can be interpreted as \textit{conservation of flow} at the vertices except the leaves.
Also, the leaves are regarded as the \textit{sources} of the flow.

Since the number of leaves is infinite a.s., the sum of the flows exiting the
leaves might be infinite. In fact, it can be seen that unimodularity implies
that the sum is always infinite a.s. The idea is to use unimodularity to quantify how \textit{large} is the flow.
Similarly, in any equivariant cut-set, the sum of the conductances of the edges is infinite a.s.
Unimodularity is also used to quantify the \textit{conductance} of an equivariant cut-set.
These are done in Definition~\ref{def:flowNorm} below.

Below, since each edge of $\bs T$ can be uniquely represented as $(v,F(v))$, the following convention is helpful.

\begin{convention}
For the vertices $v$ of $\bs T$, the symbols $\bs f(v)$ and $\bs c(v)$ are used as
abbreviations for $\bs f(v,F(v))$ and $\bs c(v,F(v))$, respectively.
Also, by $v\in \Pi$, one means that the edge $(v,F(v))$ is in $\Pi$.
\end{convention}

\begin{definition}
\label{def:flowNorm}
The \textbf{norm} of the legal equivariant flow $\bs f$ is defined as
\[
\norm{\bs f}:=\omid{\bs f(\bs o) \identity{\{\bs o\in \bs L \}}}.
\]
Also, for the equivariant cut-set $\Pi$, define
\[
\bs c(\Pi):=\omid{\bs c(\bs o)\identity{\{\bs o\in \Pi \}}} = \omid{\sum_{w\in F^{-1}(\bs o)} \bs c(w) \identity{\{w\in \Pi \}} },
\]
where the last equality follows from the mass transport principle~(I.2.2).
\end{definition}

An equivariant cut-set $\Pi$ is called \defstyle{equivariantly minimal} if there is no other
equivariant cut-set which is a subset of $\Pi$ a.s. 
If so, it can be seen that it is \defstyle{almost surely minimal} as well; i.e., in almost every realization, it is a minimal cut set (see Lemma~\ref{lem:minimalCut}).
\del{Also, it is \defstyle{almost surely minimal} if in almost every realization,
there is no subset of $\Pi$ that separates the leaves from the end. The following lemma shows that these definitions are equivalent.}


\begin{lemma}
\label{lem:maxFlowMinCut}
If $\bs f$ is a legal equivariant flow and $\Pi$ is an equivariant cut-set, then $\norm{\bs f}\leq \bs c(\Pi)$. Moreover, if
the pair $(\bs f,\Pi)$ is equivariant, then 
\[
\norm{\bs f} \leq \omid{\bs f(\bs o)\identity{\{\bs o\in \Pi \}}} \leq \bs c(\Pi).
\] 
In addition, if $\Pi$ is minimal, then equality holds in the left inequality.
\end{lemma}
\begin{proof}
One can always consider an \textit{independent coupling} of $\bs f$ and $\Pi$
(as in the proof of Theorem~\ref{thm:mdp-simple}). So assume $(\bs f, \Pi)$ is equivariant from the beginning.
Note that the whole construction (with conductances, the flow and the cut-set) is unimodular (Lemma~I.2.11).
For every leaf $v\in \bs L$, let $\bs \tau(v)$ be the first ancestor of $v$ such that $(v,F(v))\in \Pi$.
Then, send mass $\bs f(v)$ from each leaf $v$ to $\bs \tau(v)$.
By the mass transport principle~(I.2.2), one gets 
\begin{eqnarray*}
\omid{\bs f(\bs o) \identity{\{\bs o\in \bs L \}}} &=& \omid{ \identity{\{\bs o\in \Pi \}} \sum_{v\in \bs \tau^{-1}(\bs o)}\bs f(v) }\\
&\leq & \omid{ \identity{\{\bs o\in \Pi \}} \sum_{v\in D(\bs o)\cap \bs L}\bs f(v) }
= \omid{\bs f(\bs o)\identity{\{\bs o\in \Pi \}}},
\end{eqnarray*}
where the last equality holds because $\bs f$ is a flow. Moreover, if $\Pi$ is minimal,
then\del{, by Lemma~\ref{lem:minimalCut},} the above inequality becomes an equality {(see Lemma~\ref{lem:minimalCut})} and the claim follows.
\end{proof}

The main result is the following converse to the above lemma.
\begin{theorem}[Max-Flow Min-Cut for Unimodular One-Ended Trees]
\label{thm:maxFlowMinCut}
For every unimodular {marked} one-ended tree $[\bs T, \bs o;\bs c]$ equipped with conductances $\bs c$ as above, 
if $\bs c$ is bounded on the set of leaves, then 
\[
\max_f \norm{\bs f} = \inf_{\Pi} \bs c(\Pi),
\]
where the maximum is over all legal equivariant flows $\bs f$ and the infimum is over all equivariant cut-sets $\Pi$.
\end{theorem}

\begin{remark}
\label{rem:maxflow-infinite}
The claim of Theorem~\ref{thm:maxFlowMinCut} is still valid if the probability
measure (the distribution of $[\bs T, \bs o; \bs c]$) is replaced by any (possibly infinite)
measure $\mathcal P$ on $\mathcal D'_*$ supported on one-ended trees,
such that $\mathcal P(\bs o\in \bs L)<\infty$ and the mass transport principle~\xeq{I.2.2} holds.
The same proof works for this case as well. This will be used in Subsection~\ref{subsec:frostman-euclidean}.
\end{remark}

\begin{proof}[Proof of Theorem~\ref{thm:maxFlowMinCut}]
For $n\geq 1$, let $\bs T_n$ be the sub-forest of 
$\bs T$ {obtained by keeping only vertices of height at most $n$ in $\bs T$}.
Each connected component of $\bs T_n$ is a finite tree which contains some leaves of $\bs T$.
For each such component, namely $T'$, do the following: if $T'$ has more than one vertex,
consider the maximum flow on $T'$ between the
leaves and the \textit{top} vertex (i.e., the vertex with maximum height in $T'$).
If there is more than one maximum flow, choose one of them randomly and uniformly. Also, choose a
minimum cut-set in $T'$ randomly and uniformly. Similarly, if $T'$ has a single vertex $v$,
do the same for the subgraph with vertex set $\{v,F(v)\}$ and the single edge adjacent to $v$.
By doing this for all components of $\bs T_n$, a (random) function $\bs f_n$ on the edges and
a cut-set $\Pi'_n$ are obtained (by letting $\bs f_n$ be zero on the other edges).
$\Pi'_n$ is always a cut-set, but $\bs f_n$ is not a flow. However, $\bs f_n$ 
satisfies (\ref{eq:flowcons}) for vertices of $\bs T_n\setminus\bs L$ except the top vertices of the connected components of $\bs T_n$.
Also, it can be seen that $\bs f_n$ and $\Pi'_n$ are equivariant. 
 
For each component $T'$ of $\bs T_n$, the set of leaves of $T'$, excluding the top vertex,
is $\bs L\cap T'$. So the max-flow min-cut theorem of Ford-Fulkerson \cite{FoFu62} 
(see e.g., Theorem~3.1.5 of~\cite{bookBiPe17}) gives
that, for each component $T'$ of $\bs T_n$, one has 
\[
\sum_{v\in \bs L\cap T'} \bs f_n(v) = \sum_{e\in\Pi'_n\cap T'} \bs c(e).
\] 
If $u$ is the top vertex of $T'$, let $h(u)$ be the common value in the above equation. 
By using the mass transport principle~\xeq{I.2.2}
for each of the two representations of $\omid{h(\bs o)}$, one can obtain
\[
\omid{\bs f_n(\bs o) \identity{\{\bs o\in \bs L \}}}= {\omid{h(\bs o)}=} \omid{\bs c(\bs o)\identity{\{\bs o\in \Pi'_n \}}} {=\bs c(\Pi'_n)}.
\]
Since $0\leq \bs f_n(\cdot)\leq \bs c_n(\cdot)$, one can see that the distributions of $\bs f_n$ 
are tight (the claim is similar to Lemma~I.B.3 and is left to the reader).
Therefore, there is a sequence $n_1,n_2,\ldots$ and an equivariant process $\bs f'$
such that $\bs f_{n_i}\rightarrow \bs f'$ (weakly).
It is not hard to deduce that $\bs f'$ is a legal equivariant flow.
Also, since {$\bs f'(\bs o)$} and $\identity{\{\bs o\in \bs L \}}$ are
continuous functions of {$[\bs T, \bs o; \bs f']$} and their product is bounded (by the assumption on $\bs c$),
one gets that 
\[
\norm{\bs f'} {=\omid{\bs f'(\bs o)\identity{\{\bs o\in \bs L \}}} = \lim_i \omid{\bs f_{n_i}(\bs o)\identity{\{\bs o\in \bs L \}}} } = \lim_i \bs c(\Pi'_{n_i}).
\]
Therefore,
$
\max_f \norm{\bs f} \geq \inf_{\Pi} \bs c(\Pi).
$
Note that the maximum of $\norm{\bs f}$ is attained by the same tightness argument as above.
So Lemma~\ref{lem:maxFlowMinCut} implies that equality holds and the claim is proved.
\end{proof}

\subsection{A Unimodular Frostman Lemma for Point Processes}
\label{subsec:frostman-euclidean}

In the Euclidean case, another form of the unimodular Frostman lemma is given below. 
Its proof is based on the max-flow min-cut theorem of Subsection~\ref{subsec:maxflow}.
As will be seen, the claim implies that in this case, Conjecture~\ref{conj:frostman} holds up to
a constant factor (Corollary~\ref{cor:frostman-Euclidean}). However, the weight function obtained
in the theorem needs extra randomness.

\begin{theorem}
\label{thm:frostman-euclidean}
Let $\Phi$ be a point-stationary point process in $\mathbb R^k$ endowed with the Euclidean metric,
and let $\alpha\geq 0$. Then, there exists an equivariant weight function $\bs w$ on $\Phi$ such that, almost surely,
\begin{equation}
\label{eq:thm:frostman-euclidean:1}
\forall v\in \Phi,\ \forall r\geq 1: \bs w(N_r(v))\leq r^{\alpha}
\end{equation}	
and
\begin{equation}
\label{eq:thm:frostman-euclidean:2}
{\omid{\bs w(0)} \geq 3^{-k} \contentH{\alpha}{1}(\Phi).}
\end{equation}
In particular, if $\contentH{\alpha}{1}(\Phi)>0$, then $\bs w(0)$ is not identical to zero.
\end{theorem}

{In the following proof, $\Phi$ is regarded as a counting measures; i.e., for all $A\subseteq \mathbb R^d$, $\Phi(A):=\card{(\Phi\cap A)}$.}

\begin{proof}
Let $b>1$ be an arbitrary integer (e.g., $b=2$). 
For every integer {$n\geq 0$,} let $\bs Q_n$ be the stationary partition of $\mathbb R^k$
by translations of the cube $[0,b^n)^k$ as in Subsection~\ref{subsec:euclidean}.
Consider the nested coupling of these partitions for $n\geq 0$
(i.e., every cube of $\bs Q_n$ is contained in some cube of $\bs Q_{n+1}$ for every $n\geq 0$) independent of $\Phi$.
Let $\bs T_0$ be the tree whose vertices are the cubes in $\cup_n \bs Q_n$ and the edges are between
all pairs of nested cubes in $\bs Q_n$ and $\bs Q_{n+1}$ for all $n$.
Let $\bs T\subseteq \bs T_0$ be the subtree consisting of the cubes $\bs q_n(v)$ 
for all $v\in\Phi$ and $n\geq 0$. The set $\bs L$ of the leaves of $\bs T$ consists of
the cubes $\bs q_{0}(v)$ for all $v\in \Phi$. Let $\bs \sigma:=\bs q_{0}(0)\in\bs L$. 
Note that in the correspondence $v\mapsto \bs q_0(v)$, each cube $\sigma\in\bs L$ corresponds to $\Phi(\sigma)\geq 1$ points of $\Phi$. Therefore, by verifying the mass transport principle, it can be seen that the distribution of $[\bs L,\bs \sigma]$, biased by $1/\Phi(\bs \sigma)$, is unimodular; i.e., 
\[
	\omid{\frac 1{\Phi(\bs \sigma)}\sum_{\sigma'\in\bs L} g(\bs L,\bs \sigma, \sigma')} = 
	\omid{\frac 1{\Phi(\bs \sigma)}\sum_{\sigma'\in\bs L} g(\bs L,\sigma', \bs \sigma)},
\]
for every measurable $g\geq 0$.
In addition, $g$ can be allowed to depend on $\bs T$ in this equation (but the sum is still on $\sigma'\in\bs L$). Therefore, one can assume the metric on $\bs L$ is the graph-distance metric induced from $\bs T$
(see Theorem~I.3.31). Moreover, Theorem~5 of~\cite{shift-coupling} implies that by a further biasing and choosing a new root for $\bs T$, one can make $\bs T$ unimodular. More precisely, the following (possibly infinite) 
measure on $\mathcal D_*$ is unimodular:
\begin{equation}
\label{eq:thm:frostman-euclidean:P}
\mathcal P[ A]:=  \omid{\sum_{{n\geq 0}} \frac 1{e_n} \identity{A}[\bs T, {\bs q_n(0)}]},
\end{equation}
where $e_n:=\Phi(\bs q_n(0))$. 
\del{\mar{I deleted this since a bias is done as well. Keep it by explaining more?}In words, choose the root among $\bs q_{m(0)}(0),\bs q_{m(0)+1}(0),\ldots$ with the measure $\bs q_n(0) \mapsto \frac 1{e_n}$
	(which is not necessarily a probability measure).}%
Let $\mathcal E$ denote the integral operator w.r.t. the measure $\mathcal P$.
For any equivariant flow $\bs f$ on $\bs T$, the norm of $\bs f$ w.r.t. the measure $\mathcal P$
(see Remark~\ref{rem:maxflow-infinite}) satisfies
\begin{eqnarray*}
	\norm{\bs f} = \mathcal E[\bs f \cdot \identity{\bs L}]
	= \omid{\sum_{{n\geq 0}} \frac 1{e_n} \bs f(\bs q_n(0)) \identity{\{\bs q_n(0)\in \bs L \}} }
	= \omid{\frac 1{\Phi(\bs \sigma)}\bs f(\bs \sigma)},
\end{eqnarray*} 
where the second equality is by~\eqref{eq:thm:frostman-euclidean:P}.
Consider the conductance function $\bs c(\tau):=b^{n\alpha}$ for all cubes $\tau$
of edge length $b^n$ in $\bs T$ and all $n$. Therefore, Theorem~\ref{thm:maxFlowMinCut}
and Remark~\ref{rem:maxflow-infinite} imply that 
the maximum of $\omid{\bs f(\bs{\sigma})}$ over all equivariant legal flows $\bs f$ on $[\bs T, \bs{\sigma}]$
is attained (note that $[\bs T, \bs \sigma]$ is not unimodular, but the theorem can be used for $\mathcal P$).
Denote by $\bs f_0$ the maximum flow.
Let $\bs w$ be the weight function on $\Phi$ defined by $\bs w(v)=\delta \bs f_0(\bs q_{0}(v)){/\Phi(\bs q_0(v))}$,
for all $v\in \Phi$, where $\delta:=(b+1)^{-k}$.
The claim is that $\bs w$ satisfies the requirements~\eqref{eq:thm:frostman-euclidean:1} and~\eqref{eq:thm:frostman-euclidean:2}.
Since $\bs f_0$ is a legal flow, it follows that for every cube $\sigma\in\bs T$, one has
\[
\bs w(\sigma) = \delta \bs f_0(\sigma) \leq \delta \bs c(\sigma) = {\delta b^{n\alpha}}.
\]

Each cube $\sigma$ of edge length $r\in [b^n,b^{n+1})$ in $\mathbb R^k$ can be covered with at most
$(b+1)^k$ cubes of edge length $b^n$ in $\bs T_0$. If $n\geq 0$, the latter are either in $\bs T$ or do not intersect $\Phi$.
So the above inequality implies	that $\bs w(\sigma)\leq r^{\alpha}$. So \eqref{eq:thm:frostman-euclidean:1} is proved for $\bs w$.

To prove \eqref{eq:thm:frostman-euclidean:2}, given any equivariant cut-set $\Pi$ of $\bs T$,
a covering of $\Phi$ can be constructed as follows:
For each cube $\sigma\in \Pi$ of edge length say $n$, let $\bs \tau(\sigma)$ be one of the points
in $\sigma\cap\Phi$ chosen uniformly at random and put a ball of radius $b^n$ centered at $\bs \tau(\sigma)$. 
Note that this ball contains $\sigma$. Do this independently for all cubes in $\bs T$.
If a point in $\Phi$ is chosen more than once, consider only the largest radius assigned to it\todel{
	(which might be infinite)}. It can be seen that this gives an equivariant covering of $\Phi$, namely $\bs R$. One has
\begin{eqnarray*}
	\omid{\bs R(0)^{\alpha}} &\leq& \omid{\sum_{n\geq 0} b^{n\alpha} \identity{\{\bs q_n(0) \in\Pi \}} \identity{\{0 = \bs \tau(\bs q_n(0)) \}} }\\
	&=&  \omid{\sum_{n\geq 0} \frac{b^{n\alpha}}{e_n}  \identity{\{\bs q_n(0) \in\Pi \}}}.
\end{eqnarray*}

On the other hand, by~\eqref{eq:thm:frostman-euclidean:P}, one can see that 
\[
\bs c(\Pi) = \omid{\sum_{n\geq 0} \frac 1{e_n} \bs c(\bs q_n(0)) \identity{\{\bs q_n(0) \in \Pi\}}} = \omid{\sum_{n\geq 0} \frac {b^{n\alpha}}{e_n} \identity{\{\bs q_n(0) \in \Pi\}}}.
\]
Therefore, $\omid{\bs R(0)^{\alpha}} \leq \bs c(\Pi)$.
So $\contentH{\alpha}{1}(\Phi) \leq \bs c(\Pi)$. Since $\Pi$ is an arbitrary equivariant cut-set,
by the unimodular max-flow min-cut theorem established above (Theorem~\ref{thm:maxFlowMinCut})
and the maximality of the flow $\bs f_0$, one gets that 
$
\contentH{\alpha}{1}(\Phi) \leq \norm{\bs f_0} = \omid{\bs f_0(\bs \sigma)/\Phi(\bs \sigma)} = \delta^{-1} \omid{\bs w(0)}.
$
So the claim is proved.
\end{proof}

The following corollary shows that in the setting of Theorem~\ref{thm:frostman-euclidean},
the claim of Conjecture~\ref{conj:frostman} holds up to a constant factor (compare this with Lemma~\ref{lem:frostmanAuxiliary}).

\begin{corollary}
\label{cor:frostman-Euclidean}
For all point-stationary point processes $\Phi$ in $\mathbb R^k$ endowed with the Euclidean metric and all $\alpha\geq 0$, 
$
3^{-k} \contentH{\alpha}{1}(\Phi)\leq \xi^{\alpha}_1(\Phi)\leq \contentH{\alpha}{1}(\Phi).
$
\end{corollary}
\begin{proof}
The claim is directly implied by \eqref{eq:frostman-xi<H}, Theorem~\ref{thm:frostman-euclidean} and Remark~\ref{rem:frostman-ineq}.
\end{proof}

\subsection{Applications}
{The following are some {basic} applications of the unimodular Frostman lemma. This lemma is also the basis of many results of~\cite{III}.} 

\subsubsection{Cayley Graphs}
 
\begin{proposition}
	\label{prop:Cayley}
	For every finitely generated group $H$ with polynomial growth {degree} $\alpha\in [0,\infty]$, one has 
	$\xi^{\alpha}_{\infty}(H)<\infty$.
\end{proposition}
Note that if Conjecture~\ref{conj:frostman} holds, then this result implies $\measH{\alpha}(H)>0$, as conjectured in Subsection~\ref{subsec:cayley}.
\begin{proof}
	By Theorem~\ref{thm:Cayley}, $\measH{\alpha}(H)<\infty$. So the unimodular Frostman lemma (Theorem~\ref{thm:frostmanGeneral}) implies that for every $M\geq 1$, there exists $w:\dstar\to\mathbb R^{\geq 0}$ such that $w(N_M(e))\leq M^{\alpha}$ and $\omid{w(e)}=\xi^{\alpha}_M(H)$, where $e$ is the neutral element of $H$. Since the Cayley graph of $H$ is transitive and $w$ is defined up to rooted isomorphisms, $w(H,\cdot)$ is constant. Hence, $w(H,v)=\xi^{\alpha}_M(H)$ for all $v\in H$. Therefore, $\xi^{\alpha}_M(H)\card{N_M(e)} \leq M^{\alpha}$. Thus, $\xi^{\alpha}_M(H)\leq 1/c$, where $c$ is as in the proof of Theorem~\ref{thm:Cayley}. By letting $M\to \infty$, one gets $\xi^{\alpha}_{\infty}(H)\leq 1/c<\infty$.
\end{proof}

\subsubsection{Dimension of Product Spaces}
\label{subsec:product}

Let $[\bs D_1,\bs o_1]$ and $[\bs D_2,\bs o_2]$ be independent unimodular discrete metric spaces.
By considering any of the usual product metrics; e.g., the sup metric or the $p$ product metric,
the \defstyle{independent product} $[\bs D_1\times \bs D_2, (\bs o_1,\bs o_2)]$ makes sense as
a random \rooted{} discrete space. It is not hard to see that the latter is also unimodular
(see also Proposition~4.11 of~\cite{processes}). 

\begin{proposition}
\label{prop:product}
Let $[\bs D_1\times \bs D_2, (\bs o_1,\bs o_2)]$ represent the independent product of
$[\bs D_1,\bs o_1]$ and $[\bs D_2,\bs o_2]$ defined above.
Then,
\begin{equation}
\label{eq:product}
\dimH{\bs D_1}+\dimMl{\bs D_2} \leq \dimH{\bs D_1\times \bs D_2} \leq \dimH{\bs D_1}+\dimH{\bs D_2}.
\end{equation}
\end{proposition}
\begin{proof}
By Theorem~I.3.31, one can assume the metric on $\bs D_1\times \bs D_2$ is
the sup metric without loss of generality. So $N_r(v_1,v_2)=N_r(v_1)\times N_r(v_2)$.
		
The upper bound is proved first. For $i=1,2$, let $\alpha_i>\dimH{\bs D_i}$ be arbitrary.
By the unimodular Frostman lemma (Theorem~\ref{thm:frostmanGeneral}), there is a nonnegative measurable
functions $w_i$ on $\dstar$ such that $\forall v \in \bs D_i: \forall r\geq 1: w_i(N_r(v))\leq r^{\alpha} \text{, a.s.}$
In addition, $w_i(\bs o_i)\neq 0$ with positive probability. Consider the equivariant weight function
$\bs w$ on $\bs D_1\times \bs D_2$ defined by
$
\bs w(v_1,v_2):=w_1[\bs D_1,v_1]\times w_2[\bs D_2,v_2].
$
It is left to the reader to show that $\bs w$ is  an equivariant weight function.
One has $\bs w(N_r(v_1,v_2)) = w_1(N_r(v_1))w_2(N_r(v_2))\leq r^{\alpha_1+\alpha_2}$.
Also, by the independence assumption, $\bs w(\bs o_1,\bs o_2)\neq 0$ with positive probability.
Therefore, the mass distribution principle (Theorem~\ref{thm:mdp-simple})
implies that $\dimH{\bs D_1\times\bs D_2}\leq \alpha_1+\alpha_2$. This proves the upper bound.
		
For the lower bound in the claim, let $\alpha<\dimH{\bs D_1}$, $\beta<\dimMl{\bs D_2}$
and $\epsilon>0$ be arbitrary. It is enough to find an equivariant covering $\bs R$ of
$\bs D_1\times\bs D_2$ such that $\omid{\bs R(\bs o_1,\bs o_2)^{\alpha+\beta}}<\epsilon$. 
One has $\decayl{\lambda_r(\bs D_2)}>\beta$, where $\lambda_r$ is defined in~(I.3.1).
So there is $M>0$ such that $\forall r\geq M: \lambda_r(\bs D_2)<r^{-\beta}$. So for every $r\geq M$,
there is an equivariant $r$-covering of $\bs D_2$ with intensity less than $r^{-\beta}$.
On the other hand, since $\alpha <\dimH{\bs D_1}$, one has $\contentH{\alpha}{M}(\bs D_1)=0$
(by Lemma~I.3.25). Therefore there is an equivariant covering $\bs R_1$
of $\bs D_1$ such that $\omid{\bs R_1(\bs o_1)^{\beta}}<\epsilon$ and $\forall v\in \bs D_1: \bs R_1(v)\in \{0\}\cup [M,\infty)$ a.s. 
Choose the extra randomness in $\bs R_1$ independently from $[\bs D_2,\bs o_2]$. 
Given a realization of $[\bs D_1,\bs o_1]$ and $\bs R_1$, do the following: 
Let $v_1\in \bs D_1$ such that $\bs R_1(v_1)\neq 0$ (and hence, $\bs R_1(v_1)\geq M$).
One can find an equivariant subset $\bs S_{v_1}$ of $\bs D_2$ that gives a covering of
$\bs D_2$ by balls of radius $\bs R_1(v_1)$ and has intensity less than $\bs R_1(v_1)^{-\beta}$.
Do this independently for all $v_1\in \bs D_1$. Now, for all $(v_1,v_2)\in \bs D_1\times \bs D_2$, define
\[
\bs R(v_1,v_2):=\begin{cases}
\bs R_1(v_1)& \text{if } \bs R_1(v_1)\neq 0 \text{ and } v_2\in \bs S_{v_1},\\
0& \text{otherwise}.
\end{cases}
\]
Now, $\bs R$ is a covering of $\bs D_1\times \bs D_2$ and it can be seen that 
it is an equivariant covering. Also, given $[\bs D_1,\bs o_1]$ and $\bs R_1$,
the probability that $\bs o_2\in \bs S_{\bs o_1}$ is less than $\bs R_1(\bs o_1)^{-\beta}$. So one gets 
\begin{eqnarray*}
\omid{\bs R(\bs o_1,\bs o_2)^{\alpha+\beta}} 
&=& \mathbb E\Bigg[\mathbb E\Big[{\bs R(\bs o_1,\bs o_2)^{\alpha+\beta}}|{[\bs D_1,\bs o_1],\bs R_1}\Big]\Bigg]\\ 
&< & \omid{\bs R_1(\bs o_1)^{\alpha+\beta} \bs R_1(\bs o_1)^{-\beta}}
= \omid{\bs R_1(\bs o_1)^{\alpha}}
<\epsilon.
\end{eqnarray*}
So the claim is proved.
\end{proof}

The following examples provide instances where the inequalities in~\eqref{eq:product} are strict. 
\begin{example}
Assume $[\bs D_1,\bs o_1]$ and $[\bs D_2,\bs o_2]$ are unimodular discrete spaces such that
$\dimMl{\bs G_1}<\dimH{\bs G_1}$ and $\dimMl{\bs G_2}=\dimH{\bs G_2}$. By Proposition~\ref{prop:product}, one gets
\[
\dimH{\bs G_1\times \bs G_2} \geq \dimH{\bs G_1}+\dimMl{\bs G_2} > \dimH{\bs G_2}+ \dimMl{\bs G_1}.
\]
So by swapping the roles of the two spaces, an example of strict inequality in the left hand side of~\eqref{eq:product} is obtained.
\end{example}

\begin{example}
Let $J$ be a subset of $\mathbb Z^{\geq 0}$ such that $\densityU{}(J)=1$ and $\densityL{}(J)=0$ simultaneously
(see Subsection~\ref{subsec:digits} for the definitions). Let $\Psi_1$ and $\Psi_2$ be defined
as in Subsection~\ref{subsec:digits} corresponding to $J$ and $\mathbb Z^{\geq 0} \setminus J$ respectively.
Proposition~\ref{prop:digits} implies that $\dimH{\Psi_1}=\dimH{\Psi_2}=1$. On the other hand, \eqref{eq:prop:digits} implies that
\[
\card{N_{2^n}(\bs o_1\times \bs o_2)} \leq 2^{J_n+1} \times 2^{(n+1-J_n)+1} = 2^{n+3}.
\]
This implies that $\growthu{N_r(\bs o)}\leq 1$. So the unimodular Billingsley lemma
(Theorem~\ref{thm:billingsley}) implies that $\dimH{\Psi_1\times \Psi_2}\leq 1$
(in fact, equality holds by Proposition~\ref{prop:embedded} below).
So the rightmost inequality in~\eqref{eq:product} is strict here.
\end{example}

\invisible{ \textbf{Question:} What about the Minkowski dimension? We can prove that $\dimM{\bs D_1\times \bs D_2}\geq \dimM{\bs D_1}+ \dimM{\bs D_2}$ (both for lower and upper dimensions) by direct construction of coverings. \textbf{Update:} $\dimMl{\bs D_1\times \bs D_2}\geq \dimMl{\bs D_1}+ \dimMl{\bs D_2}$ but $\dimMu{\bs D_1\times \bs D_2}\geq \dimMu{\bs D_1}+ \dimMl{\bs D_2}$ (note that the latter is not $\dimMu{\bs D_2}$).
}

\invisible{\textbf{Question:} Can we relax the independence assumption? 		\textbf{Question}: Can we consider e.g., the DL-graph?  }

\subsubsection{Dimension of Embedded Spaces}
\label{subsec:embedded}

It is natural to think of $\mathbb Z$ as a subset of $\mathbb Z^2$. However, $[\mathbb Z,0]$ is
not an equivariant subspace of $[\mathbb Z^2,0]$ as defined in Definition~I.2.13.
By the following definition, $[\mathbb Z,0]$ is called \textit{embeddable in $[\mathbb Z^2,0]$}.
The dimension of embedded subspaces is studied in this subsection, which happens to be non-trivial {and requires the unimodular Frostman lemma}.

\begin{definition}
\label{def:embedded}
Let $[\bs D_0, \bs o_0]$ and $[\bs D, \bs o]$ be random \rooted{} discrete spaces. 
An \defstyle{embedding} of $[\bs D_0,\bs o_0]$ in $[\bs D, \bs o]$ is 
a (not necessarily unimodular) random \rooted{} marked discrete space $[\bs D', \bs o'; \bs m]$
with mark space $\{0,1\}$ such that
		\begin{enumerate}[(i)]
			\item $[\bs D', \bs o']$ has the same distribution as $[\bs D, \bs o]$.
			
			\item $\bs m(\bs o')=1$ a.s. and by letting $\bs S:=\{v\in \bs D': \bs m(v)=1 \}$ equipped with the
                        induced metric from $\bs D'$, $[\bs S, \bs o']$ has the same distribution
                        as $[\bs D_0,\bs o_0]$.
		\end{enumerate}
If in addition, $[\bs D_0, \bs o_0]$ is unimodular, then
$[\bs D',\bs o'; \bs m]$ is called an \defstyle{equivariant embedding} if
\begin{enumerate}[(i)]
	\setcounter{enumi}{2}
	\item \label{def:embedding:4} The mass transport principle holds on $\bs S$; i.e.,
\xeq{I.2.2} holds for functions $g(u,v):=g(\bs D',u,v; \bs m)$
such that $g(u,v)$ is zero when $\bs m(u)=0$ or $\bs m(v)=0$.
\end{enumerate}
If an embedding (resp. an equivariant embedding) exists as above, $[\bs D_0, \bs o_0]$ 
is called \defstyle{embeddable} (resp. \defstyle{equivariantly embeddable}) in $[\bs D, \bs o]$.
\end{definition}
	
It should be noted that $[\bs D', \bs o'; \bs m]$ is not an equivariant process on $\bs D$ {except in the trivial case where $\bs m(\cdot)=1$ a.s.}
	
\begin{example}
The following are instances of Definition~\ref{def:embedded}.
\begin{enumerate}[(i)]
\todel{\item \mar{I deleted this since it is already said in the beginning of the subsection.} $[\mathbb Z^n,0]$ is equivariantly embeddable in $[\mathbb Z^m,0]$ for $m\geq n$.}
			
\item Let $[\bs D_0,\bs o_0]:=[\mathbb Z,0]$ and $[\bs D, \bs o]:=[\mathbb Z^2,0]$ equipped with the sup metric.
Consider $m:\mathbb Z^2\rightarrow\{0,1\}$ which is equal to one on the boundary of the positive cone. 
Then, $[\mathbb Z^2, 0; m]$ is an embedding of $[\mathbb Z,0]$ in $[\mathbb Z^2,0]$,
but is not an equivariant embedding since it does not satisfy~\eqref{def:embedding:4}.

\item A point-stationary point process in $\mathbb Z^k$ (\rooted{} at 0) is equivariantly embeddable in $[\mathbb Z^k,0]$.

\item Let $H$ be a finitely generated group equipped with the graph-distance metric of an arbitrary
Cayley graph over $H$\del{ (see subsection~I.\ref{I-subsec:cayley})}. Then, any subgroup of $H$ {(equipped with the induced metric)}
is equivariantly embeddable in $H$.
\end{enumerate}
\end{example}
	
\todel{\mar{\ali{I propose to either delete these or move them to the end of the subsection.}}It should be noted that there are examples where $[\bs D_0, \bs o_0]$ is embeddable in $[\bs D, \bs o]$
and both are unimodular, but the former is not equivariantly embeddable in the latter. \forlater{An example can be constructed similar to Remark~\ref{rem:embedding-counter} below.}
	
\begin{remark}
$[\bs D_0, \bs o_0]$ is embeddable in $[\bs D, \bs o]$ if and only if there exists a coupling of them
such that in almost every realization, the former is a pointed subspace of the latter
(to show this, one should be cautious about the automorphisms).
In this case, it is also feasible to call $[\bs D_0, \bs o_0]$ \textit{stochastically dominated by}
$[\bs D, \bs o]$ by considering the inclusion relation between pointed metric spaces.
Being equivariantly embeddable seems to be difficult to state in this way.
\end{remark}}
	
\todel{Here is the main result of this subsection followed by some conjectures and problems.}
	
\begin{proposition}
\label{prop:embedded}
If $[\bs D_0, \bs o_0]$ and $[\bs D, \bs o]$ are unimodular discrete spaces and the
former is equivariantly embeddable in the latter, then
\begin{eqnarray}
\label{eq:thm:embedded:1}
\dimH{\bs D}&\geq& \dimH{\bs D_0},\\
\label{eq:thm:embedded:2}
\xi^{\alpha}_M(\bs D, 1) &\leq & \contentH{\alpha}{M}(\bs D_0), 
\end{eqnarray}
for all $\alpha\geq 0$ and $M\geq 1$, 
where $\xi^{\alpha}_M$ is defined in Definition~\ref{def:xi^alpha}.
\end{proposition}
	
\begin{proof}
First, assume~\eqref{eq:thm:embedded:2} holds.
For $\alpha>\dimH{\bs D}$, one has $\contentH{\alpha}{M}(\bs D)>0$ (Lemma~I.3.25).
Therefore, Lemma~\ref{lem:frostmanAuxiliary} implies that $\xi^{\alpha}_M(\bs D, 1)>0$.
Hence, \eqref{eq:thm:embedded:2} implies that $\contentH{\alpha}{M}(\bs D_0)>0$,
which implies that $\dimH{\bs D_0}\leq\alpha$. So it is enough to prove~\eqref{eq:thm:embedded:2}.
		
By the unimodular Frostman lemma (Theorem~\ref{thm:frostmanGeneral}), there is a bounded function
$w:\mathcal D_*\rightarrow\mathbb R^{\geq 0}$ such that $\omid{w(\bs o)} = \xi^{\alpha}_M(\bs D, 1)$,
and almost surely, $w(N_r(\bs o))\leq r^{\alpha}$ for all $r\geq M$. 
Assume $[\bs D', \bs o'; \bs m]$ is an equivariant embedding as in Definition~\ref{def:embedded}.
For $x\in \bs D'$, let $w'(x):=w'_{\bs D'}(x):=w[\bs D', x]$. 
Consider the random \rooted{} marked discrete space $[\bs S, \bs o'; w']$ obtained by restricting $w'$ to $\bs S$.
{By the definition of equivariant embeddings and by directly verifying the mass transport principle, the reader can obtain that $[\bs S, \bs o'; w']$ is unimodular.}
Since $[\bs S, \bs o']$ has the same distribution as
$[\bs D_0, \bs o_0]$, Proposition~I.B.1 gives an equivariant process
$\bs w_0$ on $\bs D_0$ such that $[\bs S ,\bs o'; w']$ has the same distribution as $[\bs D_0, \bs o_0; \bs w_0]$. 
According to the above discussion, one has 
$$\forall r\geq M: w'(N_r(\bs S, \bs o'))\leq w'(N_r(\bs D',\bs o')) \leq r^{\alpha}, \quad a.s.$$
This implies that $\bs w_0(N_r(\bs o_0))\leq r^{\alpha}$ a.s. Therefore, the mass distribution principle (Theorem~\ref{thm:mdp-simple})
implies that $\omid{\bs w_0(\bs o_0)}\leq \contentH{\alpha}{M}(\bs D_0)$. 
One the other hand, $$\omid{\bs w_0(\bs o_0)}= \omid{w'(\bs o')}=\omid{w(\bs o)} = \xi^{\alpha}_M(\bs D,1),$$ 
where the last equality is by the assumption on $w$.
This implies that $\contentH{\alpha}{M}(\bs D_0)\geq \xi^{\alpha}_M(\bs D,1)$ and the claim is proved.
\del{\mar{\ali{I propose to delete this paragraph and only say  `it can be seen that...'. Good?}}
It remains to prove that $[\bs S, \bs o'; w']$ is unimodular.
Since the mass transport principle holds on $\bs S$ (Definition~\ref{def:embedded}), one can show as
in Lemma~I.2.11 that the mass transport principle (I.2.2)
holds for functions $g(u,v):=g(\bs D', u,v; (w',\bs m))$ that are zero except when $\bs m(u)=\bs m(v)=1$. 
This implies the mass transport principle for $[\bs S, \bs o'; w']$. So $[\bs S, \bs o; w']$ is unimodular and the claim is proved.}
\end{proof}
	
It is natural to expect that an embedded space has a smaller Hausdorff measure. This is stated in the following conjecture.
	
\begin{conjecture}
\label{conj:embedding}
Under the setting of Proposition~\ref{prop:embedded}, for all $\alpha>0$, one has $\measH{\alpha}(\bs D)\geq \measH{\alpha}(\bs D_0)$.
\end{conjecture}
	
Note that in the case $\alpha=0$, the conjecture is implied by Proposition~I.3.28.
Also, in the general case, the conjecture is implied by~\eqref{eq:thm:embedded:2} and Conjecture~\ref{conj:frostman}. 
	

{Another problem is the validity of Proposition~\ref{prop:embedded} under the weaker assumption of being non-equivariantly embeddable.}
As a partial answer, if $\growth{\card{N_r(\bs o)}}$ exists, then~\eqref{eq:thm:embedded:1} holds. This is proved as follows: 
\begin{eqnarray*}
\dimH{\bs D_0} \leq  \essinf \growthu{\card{N_r(\bs o_0)}}
\leq  \essinf \growthu{\card{N_r(\bs o)}}\:&&\\
= \essinf \growth{\card{N_r(\bs o)}}
= \dimH{\bs D},&&
\end{eqnarray*}
where the first inequality and the last equality are implied by the unimodular Billingsley lemma (Theorem~\ref{thm:billingsley}).
	
\begin{remark}
\label{rem:embedding-counter}
Another possible way to prove Proposition~\ref{prop:embedded} and Conjecture~\ref{conj:embedding}
is to consider an arbitrary equivariant covering of $\bs D_0$ and try to extend it to an equivariant 
covering of $\bs D$ by adding some balls (without adding a ball centered at the root).
More generally, given an equivariant processes $\bs Z_0$ on $\bs D_0$, one might try to extend
it to an equivariant process on $\bs D$ without changing the mark of the root.
But {at least} the latter is not always possible. A counter example is when $[\bs D_0,\bs o_0]$ is $K_2$
(the complete graph with two vertices),  $[\bs D,\bs o]$ is $K_3$, $\bs Z_0(\bs o_0)=\pm 1$
chosen uniformly at random, and the mark of the other vertex of $\bs D_0$ is $-\bs Z_0(\bs o_0)$.
\end{remark}

\subsection{Notes and Bibliographical Comments}

The unimodular Frostman lemma (Theorem~\ref{thm:frostmanGeneral}) is analogous to  Frostman's lemma in the continuum setting (see e.g., Thm~8.17 of~\cite{bookMa95}). The proof of Theorem~\ref{thm:frostmanGeneral} is also inspired by that of~\cite{bookMa95}, but there are substantial differences. For instance, the proof of Lemma~\ref{lem:frostmanAuxiliary} and also the use of the duality of $L_1$ and $L_{\infty}$ in the proof of Theorem~\ref{thm:frostmanGeneral} are new.
The Euclidean version of the unimodular Frostman's lemma (Theorem~\ref{thm:frostman-euclidean}) and its proof are inspired by the continuum analogue (see e.g.,~\cite{bookBiPe17}).

As already explained, the unimodular max-flow min-cut theorem (Theorem~\ref{thm:maxFlowMinCut}) is inspired by the max-flow min-cut theorem for finite trees. 
Also, the results {and examples} of Subsection \ref{subsec:product} on product spaces {are inspired by analogous in the continuum setting; e.g.,} 
Theorem~3.2.1 of~\cite{bookBiPe17}.

%

\appendix
\section{Appendix}

\begin{lemma}
	\label{lem:logbound}
	Let ${(X_n)_{n=1}^{\infty}}\geq 0$ be a monotone sequence of random variables. Then almost surely, $\growthu{X_n} \leq  \growthu{\omid{X_n}}$.
	\del{\begin{eqnarray}
			\label{eq:lem:logbound:1}
			\growthu{X_n} &\leq & \growthu{\omid{X_n}}.
			\label{eq:lem:logbound:2}
			\growthl{X_n} &\leq & \growthl{\omid{X_n}}.
	\end{eqnarray}
	}
	\del{Moreover,
	if $\sum_n var(X_n)/\omid{X_n}^2 < \infty$, then 
	\[
	\growthl{X_n} = \growthl{\omid{X_n}}.
	\]}
\end{lemma}
{One can also deduce that $\growthl{X_n} \leq \growthl{\omid{X_n}}$, but this is skipped since it is not needed here.}
\begin{proof}
	The claim will be proved assuming $0\leq X_1\leq X_2\leq\cdots$.
	The non-increasing case can be proved with minor changes.
	Let $\alpha$ and $\beta$ be arbitrary such that
        $\growthu{\omid{X_n}} <\beta<\alpha$.
	So there is a constant $c$ such that  $\omid{X_n}\leq cn^{\beta}$ for all $n\geq 1$.
	Let $M:=\max\{n:X_n>n^{\alpha}\}$, with the convention $\max\emptyset:=0$.
	Below, it will be shown that $M<\infty$ a.s. Assuming this, it follows
	that $\growthu{X_n}\leq \alpha$ a.s.
	By considering this for all $\alpha$ and $\beta$, the claim is implied.
	
	Now, it is proved that $M<\infty$ a.s. With an abuse of notation, the constant $c$ below
        is updated in each step without changing the symbol.
	\begin{eqnarray*}
		\myprob{M\geq n}& =& \myprob{\exists {k\geq n}: X_k> k^{\alpha}}
		\leq  \sum_{j=0}^{\infty} \myprob{\exists k:  {n2^j\leq k\leq n2^{j+1}}, X_k> k^{\alpha}}\\	
		&\leq & \sum_{j=0}^{\infty} \myprob{X_{n2^{j+1}}> (n2^j)^{\alpha}}	
		\leq  \sum_{j=0}^{\infty} \frac{\omid{X_{n2^{j+1}}}}{(n2^j)^{\alpha}} 
		\leq  \sum_{j=0}^{\infty} \frac{c(n2^{j+1})^{\beta}}{(n2^j)^{\alpha}}\\
		&\leq & \sum_{j=0}^{\infty} c(n2^j)^{{\beta}-\alpha}
		\leq  cn^{{\beta}-\alpha}.
	\end{eqnarray*}
	The RHS is arbitrarily small for large $n$. This implies that $M<\infty$ a.s. and the claim is proved.
	\del{To prove~\eqref{eq:lem:logbound:2}, \mar{\ali{Later: double check this.}} assume $\growthl{\omid{X_n}} <\beta$. So
	there is a constant $c$ and a sequence $n_1<n_2<\cdots$ such that  $\omid{X_{n_i}}\leq cn_i^{\beta}$
	for all $i$. Define a sequence $Y_1,Y_2,\ldots$ by $Y_j=X_{n_i}$, where $i=i(j)$ is such that
	$n_i\leq j < n_{i+1}$. 
	Now, \eqref{eq:lem:logbound:1} gives 
	$
	\limsup_j {\log Y_j}/{\log j}\leq  \limsup_j {\log \omid{Y_j}}/{\log j},  
	$
	a.s.
	Note that for $n_i\leq j<n_{i+1}$, one has $\log \omid{Y_j}/ \log j \leq \log \omid{X_{n_i}}/\log n_i$.
	So the above inequality implies 
	$
	\limsup_j {\log Y_j}/{\log j}\leq  \limsup_i {\log \omid{X_{n_i}}}/{\log n_i} \leq \beta,
	$ a.s.
	where the last inequality holds by the choice of the subsequence $(n_i)_i$.
	On the other hand, since $Y_{n_i}=X_{n_i}$ for all $i$ and $Y_j$ is constant on $j\in [n_i,n_{i+1})$, one has
	\[
	\limsup_j \frac{\log Y_j}{\log j} {=} \limsup_i \frac{\log X_{n_i}}{\log n_i} \geq \liminf_n \frac{\log X_n}{\log n}.
	\]
	The above two inequalities show that $\liminf_n \log X_n / \log n \leq \beta$ a.s., which implies the claim.
	\del{
	For the third claim, assume similarly that $\omid{X_n}\geq c'n^{\beta'}$.
	Similar to above, it is enough to show that $M'<\infty$ a.s.,
	where $M':=\max\{n:X_n<n^{\alpha'} \}$ and $\alpha'<\beta'$ is arbitrary.
	\begin{eqnarray*}
		\myprob{M'\geq n}& =& \myprob{\exists {k\geq n}: X_k< k^{\alpha'}}\\
		&\leq & \sum_{j=0}^{\infty} \myprob{\exists k:  {n2^j\leq k\leq n2^{j+1}}, X_k< k^{\alpha'}}\\	
		&\leq & \sum_{j=0}^{\infty} \myprob{X_{n2^{j}}< (n2^{j+1})^{\alpha'}}\\	
		&\leq & \sum_{j=0}^{\infty} \myprob{\norm{X_{n2^j}-\omid{X_{n2^j}}}> \frac 1 2 \omid{X_{n2^j}}}\\	
		&\leq & \sum_{j=0}^{\infty} \frac {4\ \mathrm{var}(X_{n2^j})}{\omid{X_{n2^j}}^2},
	\end{eqnarray*}
	where the third inequality holds for large $n$ and fixed $\alpha$ and the last inequality
	is by Chebyshev's inequality. The assumptions imply that the last term tends to zero as $n$ tends to infinity.
        So the claim is proved.}
    }
\end{proof}

\begin{lemma}
\label{lem:BaumKatz1}
\invisible{{(Later: This lemma is implied by the LIL in Theorem~4 of [Lower functions for increasing random walks and subordinators] which is stronger (Update: It needs to assume that the tail of $S$ does not have a very slow decay. But I think that this condition is only needed for $c<\infty$ in the theorem). We have to cite. Also, I suggest to delete this lemma since we have to use the stronger result for dimension function of the image of SRW)}}
Let $X,X_1,X_2,\ldots$ be a non-negative i.i.d. sequence and $t>0$ be such that
$\myprob{X>r}\geq cr^{-t}$ for large enough $r$. Let $S_n:=X_1+\cdots+X_n$. 
Then there exists $C<\infty$ such that almost surely,
		\[
		\exists n: \forall k\geq n: S^{-1}(k) \leq Ck^{t} \log \log k. 
		\]
\end{lemma}

\begin{proof}
	First, one has
	\begin{eqnarray}
		\nonumber
		\myprob{ S^{-1}(n)\geq  m} &=& \myprob{S_{m}\leq n} \leq \myprob{\forall i\leq m: X_i\leq n} = \myprob{X\leq n}^m\\
		\label{eq:S_m<n} &\leq& (1-cn^{-t})^m \leq e^{-cmn^{-t}}.
	\end{eqnarray}
	Let $C:=2^{t+1}/c$ and $\psi(x):=Cx^t \log \log x$, 
	Therefore, for large $n$, one has 
	\begin{eqnarray*}
		\myprob{\exists k\geq n: S^{-1}(k) > \psi(k)} &=& \myprob{\max_{k\geq n} \frac{S^{-1}(k)}{\psi(k)}>1}\\
		&\hspace{-9cm}\leq & \hspace{-4.5cm} \sum_{j=0}^{\infty} \myprob{\max_{n2^j\leq k < n2^{j+1}} \frac{S^{-1}(k)}{\psi(k)}>1}
		\leq  \sum_{j=0}^{\infty} \myprob{S^{-1}(n2^{j+1}) > \psi(n2^j)}\\
		&\hspace{-9cm}\leq &\hspace{-4.5cm} \sum_{j=0}^{\infty} e^{-c\psi(n2^j)(n2^{j+1})^{-t} }
		\leq  \sum_{j=0}^{\infty} e^{-2\log \log (n2^j)}
		= \sum_{j=0}^{\infty} \frac 1{(j \log 2 + \log n)^{2}}.
	\end{eqnarray*}
	It is clear that the sum in the last term is convergent. Therefore,
        dominated convergence implies that the right hand side tends to zero as $n\rightarrow 0$. This proves the claim.
\end{proof}

\begin{lemma}
	\label{lem:regtree-weight}
	Let $\alpha<\infty$ and $(T,o)$ be a deterministic rooted tree such that
	$\mathrm{deg}(o)\geq 2$ and $\mathrm{deg}(v)\geq 3$ for all $v\neq o$.
	Let $d'$ be a metric on $T$ which is generated by \del{a function on the edges }{edge lengths} such that $d'(\cdot)\geq 1$.
	Let $w(u):=C\sum_{v \sim u} \bs d'(u,v)^{\alpha}$. Then $C=C(\alpha)$ can be chosen such that 
	for all $r\geq 0$, one has $w(N'_r(o))\geq r^{\alpha}$, {where $N'_r$ denotes the ball of radius $r$ under the metric $\bs d'$.}
	
\end{lemma}

\begin{proof}
	{Let $C$ be a constant such that $\forall x\in [0,1]: Cx^{\alpha} + (1-x)^{\alpha}\geq \frac 12$.
	It is easy to see that such $C$ exists.}
	For $r\geq 0$, let $f(r)$ be the infimum value of $w(N'_r(o))$ for all trees
	with the stated conditions. So one should prove $f(r)\geq r^{\alpha}$.
	The claim is true for $r=0$. Also, if $0<r<1$, one has $N'_r(o)=\{o\}$ and the claim is trivial. The proof uses
	induction on $\lfloor r\rfloor$. Assume that $r\geq 1$ and for all $s<\lfloor r\rfloor$, one has $f(s)\geq s^{\alpha}$.
	For $y\sim o$, let $T_y$ be the connected component containing $y$ when the edge $(o,y)$ is removed.
	It can be seen that $[T_y,y]$ satisfies the conditions of the lemma. Therefore,  
	\begin{eqnarray*}
		w(N'_r(o)) &=&  w(o) + \sum_{y: y\sim o}  w(N'_{r- d'(o,y)}(T_y, y))
		\geq 	 w(o) + \sum_{y: y\sim o} f(r- d'(o,y))\\
		&\geq & \sum_{y: y\sim o} \left[C  d'(o,y)^{\alpha} + (r- d'(o,y))^{\alpha}\right]\\
		&\geq & \mathrm{deg}(o)\cdot \min_{0\leq x \leq r}\{Cx^{\alpha} +  (r-x)^{\alpha}\}
		\geq  \mathrm{deg}(o) r^{\alpha}/2
		\geq  r^{\alpha},
	\end{eqnarray*}
	where  the third line is by the definition of $w(o)$ and the induction hypothesis, the fifth line is due to the definition of $C$,
	and the last line is by the assumption $\mathrm{deg}(o)\geq 2$.
	Hence $f(r)\geq r^{\alpha}$, which proves the induction claim.
\end{proof}

\begin{lemma}
	\label{lem:minimalCut}
	An equivariant cut-set is equivariantly minimal if and only if it is almost surely minimal.
\end{lemma}
\begin{proof}
	Let $\Pi$ be an equivariant cut-set. If $\Pi$ is almost surely minimal,
	then it is also equivariantly minimal by definition. Conversely, assume $\Pi$ is equivariantly minimal but not almost surely minimal.
	Call an edge $e'$ \textit{above} an edge $e$ if $e'$ separates $e$ from the end. Call an edge $e\in \Pi$ \textit{bad} if there is
	an edge of $\Pi$ above $e$. Let $\Pi'$ be the set of bad edges 
	of $\Pi$. Let $\Pi''$ be the set of lowest edges in $\Pi'$; i.e., the edges $e\in\Pi'$
	such that there is no other edge of $\Pi'$ below $e$. It can be seen that the assumption implies that $\Pi''$
	is nonempty with positive probability. Now, it can be seen that $\Pi\setminus\Pi''$ is an equivariant cut-set, which contradicts the minimality of $\Pi$.
\end{proof}

\del{
	{Later: Convert the table to a section*?}

	\renewcommand{\arraystretch}{1.4} 
	\begin{longtable}{c l c} 
		\caption{List of definitions and symbols from Part~I.}\\
		\label{table:symbols}
		Symbol & Description & Reference\\
		\hline
		\endfirsthead
		
		\multicolumn{3}{c}{Continuation of Table \ref{table:symbols}}\\
		
		Symbol & Description & Reference\\
		\hline
		\endhead
		 
		& unimodular discrete space& I.\ref{I-def:unimodular}\\
		& equivariant process& I.2.8\\
		& equivariant $r$-covering& I.\ref{I-def:r-covering}\\
		& equivariant covering& I.3.15\\
		$\wedge$ and $\vee$ & minimum and maximum binary operators & \\
		$\card{A}$ & number of elements in set $A$ &\\
		$\growthu{f}$ & $\limsup_{r\rightarrow\infty} {\log f(r)}/{\log r}$& I.\ref{I-def:growth}\\
		$\growthl{f}$ & $\liminf_{r\rightarrow\infty} {\log f(r)}/{\log r}$& I.\ref{I-def:growth}\\
		$\growth{f}$ & $\lim_{r\rightarrow\infty} {\log f(r)}/{\log r}$& I.\ref{I-def:growth}\\
		$\decayu{f}$ & $-\growthl{f}$& I.\ref{I-def:growth}\\
		$\decayl{f}$ & $-\growthu{f}$& I.\ref{I-def:growth}\\
		$\decay{f}$ & $-\growth{f}$& I.\ref{I-def:growth}\\
		$D$ & a discrete metric space, with elements $u, v,\ldots$&\\
		$N_r(D,v)$ & closed $r$-neighborhood of $v\in D$, with $N_0(v):=\emptyset$&\\
		$N_r(v)$ & $=N_r(D,v)$ for $v\in D$&\\
		$\dstar$ & set of equivalence classes of \rooted{} discrete spaces&\\ 
		$\mathcal D'_*$& as above for the marked case&\\ 
		$[D,o]$ & the equivalence class containing $(D,o)$&\\
		$[\bs D,\bs o]$ & a random rooted discrete space & I.\ref{I-def:randomRooted}\\
		$[\bs D,\bs o; \bs m]$& a random rooted marked discrete space& I.\ref{I-def:randomSpaceMarked}\\
		$g^+(\bs o)$ & $=\sum_v g(\bs o,v),$ outgoing mass from $\bs o$&\\
		$g^-(\bs o)$ & $=\sum_v g(v,\bs o),$ incoming mass to $\bs o$&\\
		$\intensity{\bs D}{\bs S}$& $=\myprob{\bs o\in \bs S_{\bs D}},$ intensity of the equivariant subset $\bs S$& I.\ref{I-def:subset}\\
		$\dimMu{\bs D}$& upper unimodular Minkowski dimension& I.\ref{I-def:minkowski}\\
		$\dimMl{\bs D}$& lower unimodular Minkowski dimension& I.\ref{I-def:minkowski}\\
		$\dimM{\bs D}$& unimodular Minkowski dimension& I.\ref{I-def:minkowski}\\
		$\contentH{\alpha}{1}(\bs D)$& {$\alpha$-dimensional Hausdorff content} of $\bs D$& (I.3.3)\\
		$\contentH{\alpha}{M}(\bs D)$& same, for coverings with radii in $\{0\}\cup[M,\infty)$ &  (I.3.3)\\
		$\dimH{\bs D}$& unimodular Hausdorff dimension of $\bs D$& I.3.16\\
		$\measH{\alpha}(\bs D)$& $\alpha$-dimensional Hausdorff measure of $\bs D$&I.\ref{I-def:haus-meas}\\
	\end{longtable}
}

\section*{{Acknowledgements}}
Supported in part by a grant of the Simons Foundation (\#197982) to The University of Texas at Austin.

\bibliography{bib} 
\bibliographystyle{plain}

\end{document}